\documentclass[10pt,a4paper]{amsart}
\usepackage{amsmath,amssymb, amsbsy}
\usepackage{color,psfrag}
\usepackage[foot]{amsaddr}
\usepackage[dvips]{graphicx}
\usepackage{color}
\usepackage{mathtools}
\usepackage{hyperref}
\theoremstyle{plain}

\newtheorem{theorem}{Theorem}[section]

\newtheorem{lemma}{Lemma}[section]
\newtheorem{corollary}{Corollary}[section]
\newtheorem{remark}{Remark}[section]

\newcommand{\lr} {\Delta_r}
\newcommand{\gr} {\nabla_r}
\newcommand{\hn}{\mathbb{H}^{N}}
\newcommand{\cchn}{C_c^\infty(\hn)}
\newcommand{\rgt}{|\nabla_{r}u|^2}
\newcommand{\rlt} {|\Delta_{r} u|^2}
\newcommand{\dv}{\: {\rm d}v_{\hn}}
\newcommand{\cchnm}{C_c^\infty(\hn\setminus\{o\})}

\numberwithin{equation}{section} \allowdisplaybreaks

\usepackage[text={6in,8.6in},centering]{geometry}
\begin{document}
	
	\title[On Higher order Poincar\'e inequalities with radial derivatives]{On Higher order Poincar\'e Inequalities with radial derivatives and Hardy improvements\\ on the hyperbolic space}
	
	\author[Prasun Roychowdhury]{Prasun Roychowdhury}
	\address{Department of Mathematics,
		Indian Institute of Science Education and Research,
		Dr.\ Homi Bhabha Road, Pune 411008, India}
	\email{prasunroychowdhury1994@gmail.com}
	\subjclass[2010]{26D10, 46E35, 31C12}
	\keywords{Higher order Poincar\'e inequality, Poincar\'e-Hardy inequality, Hyperbolic space}
	\date{\today}
	\maketitle

\begin{abstract}
	In this paper we  prove higher order Poincar\'e inequalities involving radial derivatives namely, 
	\begin{equation*}
	\int_{\hn}  |\nabla_{r,\hn}^{k} u|^2 \, {\rm d}v_{\hn} \geq \bigg(\frac{N-1}{2}\bigg)^{2(k-l)} \int_{\hn} |\nabla_{r,\hn}^{l} u|^2 \, {\rm d}v_{\hn} \ \ \text{ for all } u\in H^k(\hn),
	\end{equation*}
	where underlying space is $N$-dimensional hyperbolic space $\hn$, $0\leq l<k$ are integers and the constant $\big(\frac{N-1}{2}\big)^{2(k-l)}$ is sharp. Furthermore  we improve the above inequalities by adding  Hardy-type remainder terms  and the sharpness of some constants is also discussed.
       \end{abstract}

	\section{Introduction}
	Let $(M,g)$ be a Cartan-Hadamard manifold with dimension $N$ (namely, a manifold which is complete, simply connected and has everywhere non positive sectional curvature). In addition suppose Cartan-Hadamard manifolds whose sectional curvatures are bounded above by a strictly negative constant, then $M$ is known to admit a Poincar\'e inequality  which reads as follows, there exists $\Lambda>0$ such that 
	\begin{align}\label{gen_poin}
	    \int_{M} |\nabla_g u|^2 \ {\rm d}v_g \geq \Lambda \int_{M} |u|^2 \ {\rm d}v_g \text{ for all }u \in C_c^\infty(M),
	\end{align}
	where $\nabla_g$ and ${\rm d}v_g$ defines the Riemannian gradient and  volume element in $(M,g)$.
	
	\medskip 
	
	Let $\hn$ be the $N$-dimensional \emph{hyperbolic space} which is one of the most discussed Cartan-Hadamard manifolds. Indeed it enjoys all the property namely it is complete, simply connected, and has constant negative curvature.  Now for the space $\hn$, \eqref{gen_poin} holds true and $\Lambda$ turns out to be $\big(\frac{N-1}{2}\big)^2$ and moreover $\big(\frac{N-1}{2}\big)^2$ coincides with the bottom of spectrum of the Laplace-Beltrami operator on $\hn.$  
	
	\medskip
	
	Analogous to \eqref{gen_poin}, higher order Poincar\'e inequality involving higher order derivatives also holds in $\hn$. In this context, a worthy reference on this inequality is \cite[Lemma 2.4]{SK} where it has been shown that for $k$ and $l$ be non-negative integers with $0\leq l<k$ there holds 
	
	\begin{equation}\label{higher_poin}
	\int_{\hn}  |\nabla_{\hn}^{k} u|^2 \, {\rm d}v_{\hn} \geq \bigg(\frac{N-1}{2}\bigg)^{2(k-l)} \int_{\hn} |\nabla_{\hn}^{l} u|^2 \, {\rm d}v_{\hn}
	\end{equation}
	holds for all $u\in H^k(\hn)$, where 
	\begin{equation*}
	 \nabla_{\hn}^{k} :=
	\begin{dcases}
		\Delta_{\hn}^{k/2} & \text{if } k \text{ is even integer}, \\
		\nabla_{\hn} \Delta_{\hn}^{(k-1)/2}  & \text{if } k \text{ is odd integer.} \\
	\end{dcases}
	\end{equation*}
	Also $\Delta_{\hn}^k$ denotes the $k$-th iterated Laplace-Beltrami operator and $\nabla_{\hn}$ represents the Riemannian gradient in $\hn$. By constructing a minimizing sequence one can show the above constant in \eqref{higher_poin} is sharp (see \cite{NGO}).  Also it is worth to mention that the following infimum is not achieved
\begin{align*}
\underset{u\in H^k(\hn)\setminus \{0\}}{\inf} \frac{\int_{\hn} |\nabla_{\hn}^{k} u|^2 \dv }{\int_{\hn}|\nabla_{\hn}^{l} u|^2 \dv} =  \bigg(\frac{N-1}{2} \bigg)^{2(k-l)}.
\end{align*}
 This  marks an important step in the development of a  comprehensive study of inequality \eqref{higher_poin} related to its improvement.  We refer to \cite{AK,EG, EGG, MS, VHN1} for more details. Before going further, let us briefly discuss Hardy inequality and related improvements.

\medskip 

The last couple of decades have witnessed remarkable advances in the studies of  Hardy inequalities, Rellich inequalities, and related $L^2$- improvements on the Euclidean space. In this context, the seminal works of Brezis-Marcus \cite{BrezisM} and  Brezis-Vazquez \cite{Brezis} are the most significant. We recall some of the works without any claim of completeness the papers \cite{BFT,BFT2,BT,DH,FT,gaz,GM,MMP,TZ}, and the references quoted therein.  Furthermore, Carron  \cite{Carron} derived the classical Hardy inequality on Riemannian manifolds which open up new directions in the study of Hardy inequality on non-trivial geometry.  Among all the recent work in these directions,  we are bringing up only a few of them \cite{ Mitidieri2, Dambrosio, pinch, Kombe1,Kombe2,Mitidieri,YSK,LW} without a claim of completeness. A large part of these works dealt with an improvement of inequalities with \emph{optimal Hardy weight}. One of the most influential results was obtained in \cite{pinch} where \emph{optimal Hardy weight} has been derived for the general second-order elliptic operator.  

\medskip

Drawing primary
motivation from the above improvement of Hardy inequalities with $L^2$ reminder term, now one can talk about the  improvement of \eqref{gen_poin} on $\hn,$ 
by an improvement we mean here, a Hardy-type.  This has been considered in  \cite[Theorem 2.1]{EGG} related to improvement of \eqref{higher_poin} in the case $k=1$ and $l=0$ with Hardy-type remainder terms which says for $N\geq 3$ there holds 
\begin{equation}\label{or10}
	\int_{\hn} |\nabla_{\hn} u|^2 \ {\rm d}v_{\hn} -  \left( \frac{N-1}{2} \right)^{2} \int_{\hn} u^2 \ {\rm d}v_{\hn} \geq \frac{1}{4} \int_{\hn} \frac{u^2}{r^2} \ {\rm d}v_{\hn}
\end{equation}
for $u\in\cchnm$. Here $r=\rho(x,o)$ denotes the geodesic distance between a point $x$ and a fixed pole $o$ in $\hn$. Also note that both constants $\big(\frac{N-1}{2}\big)^2$ and $\frac{1}{4}$ are sharp in an obvious sense. Recently a sharper version of the above inequality considering only the \emph{radial} part of the gradient has been obtained in \cite[Theorem 2.1]{EGR} which reads as follows 

	\begin{equation}\label{r10}
	\int_{\hn} |\nabla_{r,\hn} u|^2 \ {\rm d}v_{\hn} -  \left( \frac{N-1}{2} \right)^{2} \int_{\hn} u^2 \ {\rm d}v_{\hn} \geq \frac{1}{4} \int_{\hn} \frac{u^2}{r^2} \ {\rm d}v_{\hn},
	\end{equation}
	where the constants $\big(\frac{N-1}{2}\big)^2$ and $\frac{1}{4}$ are sharp in the same sense as like in \eqref{or10}. Here $``\nabla_{r,\hn} u"$ represents the radial part of the gradient in $\hn$ and more details will be given in subsequent section. As a matter of the fact,  using \emph{Gauss's Lemma} one has $|\nabla_{\hn}u|\geq |\nabla_{r,\hn} u|$ and this readily resolves the optimality issue of the constant  $\big(\frac{N-1}{2}\big)^2$. Also note that sharpness of the constant $\frac{1}{4}$ has been proved by constructing minimizing sequence and through some delicate analysis in \cite{EGR}.
	
	\medskip

	Apart from improved ``radial Poincar\'e inequality" with Hardy term, the result in \cite[Corollary~2.3]{EGR} also dealt with the following Rellich-Poincar\'e inequality which reads as follows
	\begin{equation}\label{r21d}
	\int_{\hn} |\Delta_{r,\hn} u|^2 \ {\rm d}v_{\hn}-\left( \frac{N-1}{2} \right)^{2} \int_{\hn} |\nabla_{r,\hn} u|^2 \ {\rm d}v_{\hn}  \geq \frac{1}{4} \int_{\hn} \frac{1}{r^2}|\nabla_{r,\hn} u|^2 \ {\rm d}v_{\hn},
	\end{equation}
	where  $\left( \frac{N-1}{2} \right)^{2}$ being the sharp constant and the operator $``\Delta_{r,\hn} u"$ denotes the radial part of Laplace-Beltrami operator on the hyperbolic space (See Section~\ref{sect_est_poin}).

	\medskip

	All the above discussions mainly focus on improvement of \eqref{higher_poin} with Hardy type remainder terms, at least for the cases $k =1, l=0$ and 
	$k=2, l=1$ for the radial derivatives. So naturally one can ask whether an inequality of type \eqref{higher_poin} involving, only higher order radial derivatives holds true. Indeed the answer is affirmative, see Theorem~\ref{higher_r_poin}, namely for non-negative integers $k$ and $l$ with $0 \leq l < k,$ there holds   
	
	\begin{equation}\label{d_high_poin}
	\int_{\hn}  |\nabla_{r,\hn}^{k} u|^2 \ {\rm d}v_{\hn} \geq \bigg(\frac{N-1}{2}\bigg)^{2(k-l)} \int_{\hn} |\nabla_{r,\hn}^{l} u|^2 \ {\rm d}v_{\hn} \ \ \text{ for all } u\in H^k(\hn).
	\end{equation}

	 Moreover 
	 the constant $\big(\frac{N-1}{2}\big)^{2(k-l)}$ is sharp which is intercepted by the delicate use of integral representation of the volume of a ball in hyperbolic space $\hn$ and by using some clever estimates derived in \cite{NGO}. For the uniformity of the work, we shall discuss all the things briefly in Section \ref{sect_est_poin}.

	\medskip

 The next aim of this article is to establish improved Hardy-type inequality 	associated with  above radial higher order Poincar\'e inequality \eqref{d_high_poin}. This improvement is very much align with \cite[Theorem 2.1]{EG}. We briefly recall the theorem which reads 
	  as for integer $k,l$ with $0\leq l<k$ and $N>2k$, then there exist $k$ positive constants $\alpha^{j}_{k,l}=\alpha^{j}_{k,l}(N)$ such that following inequality holds 
	\begin{align}\label{imp_high_poin}
	\int_{\hn}  |\nabla_{\hn}^{k} u|^2 \, {\rm d}v_{\hn} & \geq \bigg(\frac{N-1}{2}\bigg)^{2(k-l)} \int_{\hn} |\nabla_{\hn}^{l} u|^2 \, {\rm d}v_{\hn}+ \alpha^{1}_{k,l} \int_{\hn}\frac{u^2}{r^{2}}\dv\\ \notag &+ \sum_{j=2}^{k-1}\alpha^{j}_{k,l}\int_{\hn}\frac{u^2}{r^{2j}}\dv+\alpha^{k}_{k,l}\int_{\hn}\frac{u^2}{r^{2k}}\dv.
	\end{align}
	Also note that $\alpha^{1}_{k,l}$ and $\alpha^{k}_{k,l}$ signifies the coefficient for the leading term as $r\rightarrow 0$ and $r\rightarrow \infty$. In the same spirit we would like to obtain the above inequality involving only the higher order radial derivatives. Of course it is not at all a straight forward generalization of above. We need to devise all together a new strategy to obtain our desired results. To this end, we briefly explain the result we obtain in Section~\ref{sect_improve},  for non-negative integer $k$ and $l$ with $0\leq l<k$ for $N>2k$ there exist $k$ positive constants $C_{k,l}^j$ such that 
	\begin{align}\label{imp_r_high_poin}
	\int_{\hn}  |\nabla_{r,\hn}^{k} u|^2 \, {\rm d}v_{\hn}  - \bigg(\frac{N-1}{2}\bigg)^{2(k-l)} \int_{\hn} |\nabla_{r,\hn}^{l} u|^2 \, {\rm d}v_{\hn}\geq \sum_{j=1}^{k}C^{j}_{k,l}\int_{\hn}\frac{u^2}{r^{2j}}\dv.
	\end{align}
	In a similar fashion, like in \cite{EG}, here also we calculate the explicit expression of $C_{k,l}^1$ and $C_{k,l}^k$ related to dominating term for $r\rightarrow 0$ and $r \rightarrow \infty$ respectively.  Moreover we obtain another version of improvement  \eqref{higher_r_poin} with dimension restriction $N\geq 4k-1$ but with a better constants in front of the leading order Hardy term. We refer the readers to Section \ref{sect_improve} for detail study of these inequalities. 	
	
	\medskip
	
	The article is organized as follows: Section \ref{sect_est_poin} is devoted to set up higher order Poincar\'e inequality in terms of higher order radial derivatives involving the Riemannian gradient and Laplace-Beltrami operator on the hyperbolic space $\hn$. In this section, we also discuss the optimality issue. The next Section \ref{sect_lemma} is devoted to some of the key lemmas to prove Hardy-type improvements. In the last Section \ref{sect_improve} we prove \eqref{imp_r_high_poin} with detailed description of coefficient related to asymptotic Hardy type remainder terms. Broadly speaking here we deal the general case of integers by separating it into odd-even cases and exploiting induction as a key ingredient to reach towards our result.
	
    \medskip

	\section{Higher order Poincar\'e Inequality with radial derivatives}\label{sect_est_poin}
	
	The $N$-dimensional hyperbolic space $\hn$ admits \it Riemannian Model \rm manifold structure whose metric $g$ is represented in spherical coordinates as follows
	\begin{equation*}
	{\rm d}s^2 = {\rm d}r^2 + \psi^2(r) \, {\rm d}\omega^2,
	\end{equation*}
	where ${\rm d}\omega^2$ is the metric on sphere $\mathbb{S}^{N-1}$ and $\psi(r)=\sinh r$. We refer the readers to \cite{RGR,PP} for more details about Riemannian Model manifold structure. By means of this structure of $\hn$, it admits polar coordinate transformation and we can write $x=(r,\sigma)\in(0,\infty)\times\mathbb{S}^{N-1}$ which will be used several times in this article. Let us first define the following two quantities. For  $u\in C_c^\infty(\hn)$ we write
	\begin{align*}
	\Delta_{r,\hn}u := \frac{\partial^2 u}{\partial r^2}+(N-1)\coth r \frac{\partial u}{\partial r}\quad \text{ and } \quad
	\nabla_{r,\hn}u := \big(\frac{\partial u}{\partial r},0\big).
	\end{align*}
		These two quantities are so called radial contribution of the Laplace-Beltrami operator and Riemannian gradient in $\hn$ respectively. For notational economy we will always use $\Delta_{r,\hn}=\Delta_r, \ \nabla_{r,\hn}=\nabla_{r}$ and finally for any non-negative integer $k$ we denote $\nabla_{r,\hn}^{k}=\nabla_{r}^k$, which is described as
	\begin{equation*}
	\nabla_{r,\hn}^{k} :=
	\begin{dcases}
	\Delta_{r,\hn}^{k/2} & \text{if } k \text{ is even integer}, \\
	\nabla_{r,\hn} \Delta_{r,\hn}^{(k-1)/2}  & \text{if } k \text{ is odd integer.} \\
	\end{dcases}
	\end{equation*}
	
	Before proving the main result we want to verify one useful tool in Partial Differential Equation namely, \emph{integration by parts} formula.
		\begin{lemma}\label{lemma_1}
			Let $f$ and $g \in \cchn$. Then it holds
			\begin{equation*}
				\int_{\hn} (\Delta_{r} f)\ g \: {\rm d}v_{\hn} = -\int_{\hn} (\nabla_r f) \cdot (\nabla_r g) \: {\rm d}v_{\hn} = \int_{\hn} f\ (\Delta_{r} g)\: {\rm d}v_{\hn}.
			\end{equation*}
		\end{lemma}
		\begin{proof}
			Exploiting polar coordinate transformation and by parts formula on first variable i.e., in radial coordinate we deduce
			\begin{align*}
				& \int_{\hn} (\Delta_{r} f)\ g \: {\rm d}v_{\hn}=\int_{\mathbb{S}^{N-1}}\int_{0}^{\infty}\bigg(\frac{\partial^2 f}{\partial r^2}+(N-1)\frac{\psi^{\prime}}{\psi} \frac{\partial f}{\partial r}\bigg)\psi^{(N-1)}\ g \: {\rm d}r \ {\rm d}\sigma\\  &=-\int_{\mathbb{S}^{N-1}}\int_{0}^{\infty} \frac{\partial f}{\partial r} \frac{\partial g}{\partial r} \psi^{(N-1)} \: {\rm d}r \ {\rm d}\sigma= -\int_{\hn} (\nabla_r f) \cdot (\nabla_r g) \: {\rm d}v_{\hn} \\& =\int_{\mathbb{S}^{N-1}}\int_{0}^{\infty} f \frac{\partial^2 g}{\partial r^2}\psi^{(N-1)} \: {\rm d}r \ {\rm d}\sigma+(N-1)\int_{\mathbb{S}^{N-1}}\int_{0}^{\infty}f\frac{\psi^{\prime}}{\psi} \frac{\partial g}{\partial r}\psi^{(N-1)} \: {\rm d}r \ {\rm d}\sigma = \int_{\hn} f\ (\Delta_{r} g)\: {\rm d}v_{\hn}.
			\end{align*}
		\end{proof}
	
	We are now ready to discuss one of our main result.	
	\begin{theorem}\label{higher_r_poin}
		 For all non-negative integers $l \text{ and } k$ with $0\leq l<k$ and for all $N\geq 3$ there holds 
		\begin{equation}\label{eq_main_th_1}
		\int_{\hn}  |\nabla_{r}^{k} u|^2 \, \dv \geq \bigg(\frac{N-1}{2}\bigg)^{2(k-l)} \int_{\hn} |\nabla_{r}^{l} u|^2 \, \dv \ \text{ for all } u\in H^k(\hn).
		\end{equation}
	 Also the constant $\big(\frac{N-1}{2}\big)^{2(k-l)}$ is optimal in a sense that no inequality of the form 
	 \begin{equation*}
		\int_{\hn}  |\nabla_{r}^{k} u|^2 \, \dv \geq \Lambda\int_{\hn} |\nabla_{r}^{l} u|^2 \, \dv
		\end{equation*}
		holds, for $u\in H^k(\hn)$, when $\Lambda > \big(\frac{N-1}{2}\big)^{2(k-l)}$.
	\end{theorem}
	
	\begin{proof}
	We divide the proof into three steps. In the first step, we show the existence of the inequality \eqref{higher_r_poin} and in the rest of the two steps, we will tackle the optimality issues.
	
	{\bf Step 1.} Beginning with $u\in\cchn$, for the case $k=1$ and $l=0$, we already have from \eqref{r10}
	\begin{equation}\label{use_1}
	\int_{\hn} \rgt \ {\rm d}v_{\hn} \geq \left( \frac{N-1}{2} \right)^{2} \int_{\hn} u^2 \ {\rm d}v_{\hn}.
	\end{equation}
	
	Now we will arrive to the higher order Poincar\'e inequality in terms of radial derivatives by using Lemma \ref{lemma_1} and H\"older inequality step by step.
	\begin{align*}
	  &\int_{\hn} \rgt\dv   = \int_{\hn} (\nabla_r u) \cdot (\nabla_r u) \: {\rm d}v_{\hn} = -\int_{\hn} (\Delta_{r} u)\ u \: {\rm d}v_{\hn}=-\int_{\hn} (\nabla_{r}^{2} u )\ u \: {\rm d}v_{\hn}\\ & \leq \bigg(\int_{\hn}|\nabla_{r}^{2} u|^2\dv\bigg)^{\frac{1}{2}}\bigg(\int_{\hn}u^2\dv\bigg)^{\frac{1}{2}}\leq \frac{2}{(N-1)} \bigg(\int_{\hn}|\nabla_{r}^{2} u|^2\dv\bigg)^{\frac{1}{2}} \bigg(\int_{\hn}\rgt\dv\bigg)^{\frac{1}{2}}.
	\end{align*}
	
	By arranging these we deduce the case $k=2$ and $l=1$, which reads as
	\begin{equation}\label{use_2}
	\int_{\hn}|\nabla_{r}^{2} u|^2\dv\geq \bigg(\frac{N-1}{2}\bigg)^{2}\int_{\hn} \rgt\dv.
	\end{equation}
	
	Now we are ready to prove the higher order Poincar\'e inequality in terms of radial derivatives 
	 using induction. Suppose $k$ be an even integer with $k\geq 2$, then using \eqref{use_1} we get
	\begin{align*}
	& \int_{\hn}|\gr^k u|^2 \dv  = \int_{\hn}|\lr^{k/2} u|^2 \dv \\ & \leq \bigg(\frac{N-1}{2}\bigg)^{-2}\int_{\hn}|\gr \lr^{k/2} u|^2\dv  = \bigg(\frac{N-1}{2}\bigg)^{-2}\int_{\hn}|\gr^{k+1} u|^2 \dv.
	\end{align*}

	Assume $k$ be an odd integer with $k\geq 3$, then exploiting \eqref{use_2} we have
	\begin{align*}
	& \int_{\hn}|\gr^k u|^2 \ {\rm d}v_{\hn} = \int_{\hn}|\gr \lr^{\frac{k-1}{2}} u|^2 \ {\rm d}v_{\hn}  \\ & \leq \bigg(\frac{N-1}{2}\bigg)^{-2}\int_{\hn}|\gr^2 \lr^{\frac{k-1}{2}} u|^2 \ {\rm d}v_{\hn}  = \bigg(\frac{N-1}{2}\bigg)^{-2} \int_{\hn}|\gr^{k+1} u|^2 \ {\rm d}v_{\hn}.
	\end{align*}
	
	Finally use of \eqref{use_1} and \eqref{use_2} over and over yields the general case
	\begin{align*}
	& \int_{\hn} u^2\dv  \leq \bigg(\frac{N-1}{2}\bigg)^{-2}\int_{\hn}|\gr u|^2 \dv  \leq \bigg(\frac{N-1}{2}\bigg)^{-4}\int_{\hn}|\nabla_{r}^{2} u|^2\dv\\ & \leq \bigg(\frac{N-1}{2}\bigg)^{-6} \int_{\hn}|\nabla_{r}^{3} u|^2\dv \leq \cdots \leq \bigg(\frac{N-1}{2}\bigg)^{-2k} \int_{\hn}|\nabla_{r}^{k} u|^2\dv.
	\end{align*}
	
	Now if we commence with any non-negative integer $k \text{ and } l$ with $k>l$, then beginning with $\int_{\hn}|\gr^l u|^2\dv$ and repeatedly exploiting $\eqref{use_1}$ and $\eqref{use_2}$, we will get to $\int_{\hn}|\gr^k u|^2\dv$ with appropriate constant. At the end density arguments establish the result \eqref{higher_r_poin}.
	
	\medskip
	
	{\bf Step 2.} In the rest of the section we discuss the optimality issues. The argument runs similar like the proof of sharpness of constant in \cite{NGO}. Radial behaviour of the operator on a radial function is crucially used here. Let us write the \emph{integral representation} of the volume of a ball in hyperbolic space $\hn$ as follows 
	\begin{align}
	G(r):=N\omega_N\int_{0}^{r}(\sinh s)^{N-1} \: {\rm d}s,
	\end{align}
	where $\omega_N$ denotes the surface measure of unit sphere $\mathbb{S}^{N-1}$ in the underlying $N$-dimensional Euclidean space $\mathbb{R}^N$. Observe that this function $G(r):[0,\infty)\rightarrow[0,\infty)$ defines the hyperbolic volume of the ball with center at fixed pole $o$ and radius $r=\rho(o,x)$ i.e. $G(r):=\text{Vol}(B(o\ ; \ \rho(o,x)))$. Note that $G(r)$ is clearly continuous and strictly increasing function.  Next choose $F(r)$ as inverse of $G(r)$ and it's clear that $F(r)$ will be continuous, strictly increasing function and satisfying
	\begin{align}\label{help_eq}
	r=N\omega_N\int_{0}^{F(r)}(\sinh s)^{N-1}\: {\rm d}s \text{ for }r\geq 0.
	\end{align}
	
    In the above \eqref{help_eq} using $(\sinh s)\leq (\cosh s)$ and exploiting L'Hospital's rule we deduce for any non-negative real number there hold
	\begin{align*}
	(N-1)r\leq N\omega_N(\sinh F(r))^{N-1} \text{\  and \ } \underset{r\rightarrow\infty}{\lim} \frac{N\omega_N(\sinh F(r))^{N-1}}{(N-1)r}=1.
	\end{align*}
	
	So by the definition of limit we can say that for any $\epsilon>0$ there exist real number $R_0$ such that whenever $r\geq R_0$ there holds
	\begin{align}\label{help_eq_2}
	(N-1)r\leq N\omega_N(\sinh F(r))^{N-1}\leq (1+\epsilon)(N-1)r.
	\end{align}
	
	Now for $R>R_0$, let us define the radial function $f_R:[0,\infty)\rightarrow [0,\infty)$ as follows 
	\begin{equation*}
	f_R(r) :=
	\begin{dcases}
	R_0^{-\frac{1}{2}} & \text{if } r\in[0,R_0), \\ r^{-\frac{1}{2}} & \text{if } r\in [R_0,R),\\
	R^{-\frac{1}{2}}\big(2-\frac{r}{R}\big)  & \text{if } r\in[R,2R), \\
	0 & \text{if } r\in[2R,\infty). \\
	\end{dcases}
	\end{equation*}
	
	Along with this function we define two more sequences of radial functions $\{v_{R,i}\}_{i\geq 0}$ and $\{g_{R,i}\}_{i\geq 1}$ for $i\geq0$ in below : \\ 
	(i) first define $v_{R,0}(r):=f_R(r)$; \\
	(ii) next construct the maximal function $g_{R,i+1}:=\frac{1}{r}\int_{0}^{r} v_{R,i}(t) \ {\rm d}t$;\\
	(iii) finally we set $v_{R,i+1}:=\int_{r}^{\infty} \frac{t\, g_{R,i+1}(t)}{(N\omega_N (\sinh F(t))^{N-1})^2} \ {\rm d}t$.
	
	\medskip
	
	These two non-increasing functions $v_{R.i}$ and $g_{R,i}$ can be computed explicitly. We are skipping the details here. Without giving the proof, we are mentioning a key lemma which will crucially play an important role here. For details refer to \cite[Proposition 2.1]{NGO}. 
	
	\begin{lemma}\label{help_lemma}
		For any $\epsilon>0$ and $i\geq 1$. there exist radial functions $h_{R,i}$ and $w_{R,i}$ such that the  following holds\\
		(i) $v_{R,i}=h_{R,i}+w_{R,i};$\\
		(ii) there exist positive real number $C$ independent of $R$ such that $\int_0^{\infty} |w_{R,i}|^2 \ {\rm d}s \leq C;$\\
		(iii) and $\frac{1}{(1+\epsilon)^{2i}}\big(\frac{2}{N-1}\big)^{2i}f_R\leq h_{R,i}\leq \big(\frac{2}{N-1}\big)^{2i}f_R.$
	\end{lemma}

\medskip
	
	{\bf Step 3.} 
	Let us define the radial function in terms of $f_R$,
	\begin{align*}
	u_R(x):=f_R(\text{Vol}(B(o\ ; \ \rho(o,x)))).
	\end{align*} 
	Now we will compute $\int_{\hn}|u_R|^2\dv$ and $\int_{\hn}|\gr u_R|^2\dv$ separately and finiteness of those quantities will confirm, $u_R\in H^{1}(\hn)$. Shifting into polar coordinate and by exploiting change of variable it follows,
	\begin{align*}
	&\int_{\hn}|u_R(x)|^2\dv=\int_{0}^{\infty}\int_{\mathbb{S}^{N-1}}|u_R(r,\sigma)|^2 (\sinh r)^{N-1} \ {\rm d}\sigma \ {\rm d}r\\&=N\omega_N\int_{0}^{\infty} (f_R(G(r)))^2(\sinh r)^{N-1}{\rm d}r=\int_{0}^{\infty}(f_R(t))^2\ {\rm d}t = \ln \bigg(\frac{R}{R_0}\bigg) +\frac{4}{3}.
	\end{align*}
	
	Use of \eqref{help_eq_2} yields
	\begin{align*}
	& \int_{\hn}|\gr u_R(x)|^2 \ {\rm d}v_{\hn}=\int_{0}^{\infty}\int_{\mathbb{S}^{N-1}}\big|\frac{\partial}{\partial r}u_R(r,\sigma)\big|^2 (\sinh r)^{N-1} \ {\rm d}\sigma \ {\rm d}r\\ & =N\omega_N\int_{0}^{\infty} (f_R'(G(r)) G'(r) )^2(\sinh r)^{N-1} \ {\rm d}r\\&=(N\omega_N)^3\int_{0}^{\infty}(f_R'(G(r)))^2(\sinh r)^{3(N-1)} \ {\rm d}r=(N\omega_N)^2\int_{0}^{\infty} (f_R'(t))^2(\sinh F(t))^{2(N-1)}  \ {\rm d}t\\& \leq  (1+\epsilon)^2 (N-1)^2 \int_{0}^{\infty} (f_R'(t))^2 t^2 \ {\rm d}t = \frac{(N-1)^2}{4} (1+\epsilon)^2 \bigg[\ln \bigg(\frac{R}{R_0}\bigg)+\frac{28}{3}\bigg].
	\end{align*}
	
	Next considering the ratios of these two quantities we deduce 
	\begin{align*}
	\underset{u\in H^1(\hn)\setminus \{0\}}{\inf} \frac{\int_{\hn} |\gr u|^2 \dv }{\int_{\hn} |u|^2 \dv} \leq \underset{R\rightarrow \infty}{\liminf} \frac{\int_{\hn} |\gr u_R|^2 \dv }{\int_{\hn} |u_R|^2 \dv} \leq \bigg(\frac{N-1}{2} \bigg)^2(1+\epsilon)^2.
	\end{align*}
	
	So from the existence inequality \eqref{eq_main_th_1}, for the case $k=1,l=0$ and letting $\epsilon$ towards zero we can conclude $\big(\frac{N-1}{2} \big)^2$ is optimal constant. It's worth noticing that, by the help of \emph{Gauss's Lemma} one can quickly infer that $\big(\frac{N-1}{2} \big)^2$ is the best constant whenever $k=1$ and $l=0$ but this method will help us to comment about the optimality of the other higher index cases.
	
	Next we will deal with the case $k=2$ and $l=0$. In this context we define
	\begin{align*}
	u_R(x):=v_{R,1}(\text{Vol}(B(o\ ; \ \rho(o,x)))).
	\end{align*}
	
	Due to the radial behaviour of $u_R(x)$ and by the definition of $v_{R,1}$ we can write
	\begin{align*}
	-\Delta_{r}u_R=-\Delta_{\hn} u_R=f_R(\text{Vol}(B(o\ ; \ \rho(o,x)))).
	\end{align*}
	
	Now like earlier we have that
	\begin{align*}
	\int_{\hn} |\lr u_R|^2\dv= \int_{\hn}|f_R(\text{Vol}(B(o\ ; \ \rho(o,x))))|^2\dv=\ln \bigg(\frac{R}{R_0}\bigg) +\frac{4}{3}.
	\end{align*}
	
	By the help of the Lemma \ref{help_lemma}, polar coordinate transformation and change of variable we deduce 
	\begin{align*}
	\big(\int_{\hn} |u_R(x)|^2\dv \big)^{\frac{1}{2}} =  \big(\int_{0}^{\infty} |v_{R,1}(r)|^2 {\rm d}r \big)^{\frac{1}{2}}& \geq \big(\int_{0}^{\infty} |h_{R,1}(r)|^2 {\rm d}r \big)^{\frac{1}{2}}-\big(\int_{0}^{\infty} |w_{R,1}(r)|^2 {\rm d}r \big)^{\frac{1}{2}} \\& \geq \frac{4}{(1+\epsilon)^2(N-1)^2}\bigg[\ln \bigg(\frac{R}{R_0}\bigg) +\frac{4}{3}\bigg]^{\frac{1}{2}}-C.
	\end{align*}
	
	Again this implies
	\begin{align*}
	\underset{u\in H^2(\hn)\setminus \{0\}}{\inf} \frac{\int_{\hn} |\lr u|^2 \dv }{\int_{\hn} |u|^2 \dv} \leq \underset{R\rightarrow \infty}{\liminf} \frac{\int_{\hn} |\lr u_R|^2 \dv }{\int_{\hn} |u_R|^2 \dv} \leq \bigg(\frac{N-1}{2} \bigg)^4(1+\epsilon)^4.
	\end{align*}
	
	So the inequality $\eqref{eq_main_th_1}$ and letting $\epsilon$ towards zero we obtain $\big(\frac{N-1}{2} \big)^4$ is the best constant for the case $k=2$ and $l=0$. Recall that, exploiting \cite[Lemma 6.1]{EGR} we can tackle the optimality issue but once again this method will help to speak about the other sharpness cases. 
	
	Now consider the case $k=2m,l=0$ and we define 
	\begin{align*}
	u_R(x):= v_{R,m}(\text{Vol}(B(o\ ; \ \rho(o,x)))).
	\end{align*}
	
	Again due to the radial nature of the function it is easy to see that 
	\begin{align*}
	(-\Delta_{r})^m u_R(x)=(-\Delta_{\hn})^m u_R(x)=f_R(\text{Vol}(B(o\ ; \ \rho(o,x)))).
	\end{align*}
	
	Exploiting Lemma \ref{help_lemma} for the case $i=m$ and running similar argument like earlier we deduce the constant $\big(\frac{N-1}{2} \big)^{4m}$ is best possible.
	
	Now consider the case $k=2m+1,l=0$ and if possible assume there exist a constant $\Theta$ such that for $u\in\cchn$ there holds
	\begin{align*}
	\Theta \int_{\hn}|u|^2 \dv \leq \int_{\hn} |\gr(\lr^m u)|^2 \dv.
	\end{align*}
	But from earlier evaluation, we know the constant is sharp for the case $k=2m+2,l=0$. So using this and inequality \eqref{eq_main_th_1} for the case $k=2,l=1$ we can write 
	\begin{align*}
	\Theta \int_{\hn}|u|^2 \dv \leq \int_{\hn} |\gr(\lr^m u)|^2 \dv\leq \bigg(\frac{N-1}{2}\bigg)^{-2} \int_{\hn} |\lr^{m+1} u|^2 \dv.
	\end{align*}
	This implies
	\begin{align*}
	\Theta\bigg(\frac{N-1}{2}\bigg)^{2}\leq \bigg(\frac{N-1}{2}\bigg)^{2m+4} \implies \Theta\leq \bigg(\frac{N-1}{2}\bigg)^{2m+2}.
	\end{align*}
	This and density argument proves that, for the case $k=2m+1$ and $l=0$, whenever $u\in H^{2m+1}(\hn)$, the constant $\big(\frac{N-1}{2}\big)^{2m+2}$ is optimum. Hence by the same technique we can prove that, constant $\big(\frac{N-1}{2}\big)^{2(k-l)}$ is sharp for any  non-negative integer $k$ and $l$ with $k>l$, whenever $u\in H^k(\hn)$.
	\end{proof}
	
	From the result in Theorem \ref{higher_r_poin}, we can write
	\begin{align*}
    \underset{u\in H^k(\hn)\setminus \{0\}}{\inf} \frac{\int_{\hn} |\gr^k u|^2 \dv }{\int_{\hn}|\gr^l u|^2 \dv} =  \bigg(\frac{N-1}{2} \bigg)^{2(k-l)}
    \end{align*}
	and for this reason, always strict inequality holds in \eqref{eq_main_th_1}, except $u=0$. So this observation opens the account for improvement of \eqref{eq_main_th_1} and to support this we  proceed to the subsequent  sections.
	
	\medskip
	\section{Preparatory Results for Improvement via Hardy-type Remainder terms}\label{sect_lemma}
	
		In this section, we will mainly focus on some useful lemmas which will help to construct improvement of \eqref{eq_main_th_1}. We want to point out that \emph{Spherical decomposition} is the key method in the first part of this section. Over the years this method has become a remarkable tool in functional inequality. So before going further first briefly recall some useful facts about this method (for details refer to \cite{MSS}).
		
		Start with $u(x)=u(r,\sigma)\in \cchnm$ where $r\in[{0},\infty)$ and $\sigma\in \mathbb{S}^{N-1}$. From \cite[Ch. 4, Lemma 2.18]{ES}, we can write
	\begin{equation*}
		u(x):=u(r,\sigma)=\sum_{n=0}^{\infty}d_n(r)P_n(\sigma)
	\end{equation*}
	in the Hilbert space $L^2(\hn)$, where $\{ P_n \}$ is an orthonormal system of spherical harmonics in the space $L^2(\mathbb{S}^{N-1})$ and 
	\begin{equation*}
		d_n(r)=\int_{\mathbb{S}^{N-1}}u(r,\sigma)P_n(\sigma) \ {\rm d}\sigma\,.
	\end{equation*} 
	
	Moreover, $P_n$ is called spherical harmonics of order $n$ and this is the restriction to $\mathbb{S}^{N-1}$ of a homogeneous $n$ degree harmonic polynomial. The first application of this method will be the establishment of weighted Hardy inequality in terms of radial derivatives. For a similar type of result, one can refer to \cite[Theorem 5.1]{BGGP}.
	
	\begin{theorem}\label{th_1}
		Assume that $0\leq 2 \alpha < (N+3)$. For all $u\in \cchnm$, there holds 
		\begin{align}\label{eq_th_1}
		& \int_{\hn} \frac{\rgt}{r^{\alpha}} \ {\rm d}v_{\hn} \, \geq  \,
		\frac{(N-2-\alpha)^2}{4} \int_{\hn} \frac{u^2}{r^{\alpha+2}} \ {\rm d}v_{\hn}\\ \notag & +\frac{(N-1)}{2}\int_{\hn} \frac{u^2}{r^{\alpha}} \ {\rm d}v_{\hn} + \frac{(N-1)(N-3-2\alpha)}{4} \int_{\hn} g(r)\frac{u^2}{r^{\alpha}} \ {\rm d}v_{\hn},
		\end{align}
		where $g(r)=\frac{r\coth r -1}{r^2}$ is a positive function. Moreover, the constant $\frac{(N-2-\alpha)^2}{4}$ is optimal in the obvious sense.
	\end{theorem}
	
	\begin{proof}
		Start with $u\in \cchnm$ and we define
		\begin{equation*}
		v(x)=(\sinh r)^{(N-1)/2}u(x)r^{-\alpha/2} \text{ where } x=(r,\sigma)\in (0,\infty)\times \mathbb{S}^{N-1}.
		\end{equation*}
		
		An easy calculation gives
		\begin{equation*}
		\frac{1}{r^{\alpha/2}}\frac{\partial u }{\partial r}=\frac{1}{(\sinh r)^{\frac{(N-1)}{2}}}\bigg[\frac{\partial v }{\partial r} -\frac{(N-1)}{2}(\coth r)\:v+\frac{\alpha}{2}\frac{v}{r}\bigg].
		\end{equation*}
		
		After squaring the above term, we observe
		\begin{align*}
		\int_{\hn} \frac{\rgt}{r^{\alpha}} \ {\rm d}v_{\hn} &= \int_{\hn}\frac{1}{(\sinh r)^{(N-1)}}\bigg(\frac{\partial v}{\partial r}\bigg)^2 \ {\rm d}v_{\hn}+\frac{(N-1)^2}{4}\int_{\hn}\frac{1}{(\sinh r)^{(N-1)}}(\coth r)^2\: v^2 \ {\rm d}v_{\hn}\\&+\frac{\alpha^2}{4}\int_{\hn}\frac{1}{(\sinh r) ^{(N-1)}}\frac{ v^2}{ r^2} \ {\rm d}v_{\hn}-(N-1)\int_{\hn}\frac{1}{(\sinh r)^{(N-1)}}\frac{\partial v }{\partial r} (\coth r)\:v \ {\rm d}v_{\hn}\\&+\alpha\int_{M}\frac{1}{(\sinh r)^{(N-1)}}\frac{\partial v }{\partial r} \frac{v}{r} \ {\rm d}v_{\hn}-\frac{\alpha(N-1)}{2}\int_{M}\frac{1}{(\sinh r)^{(N-1)}} (\coth r)\frac{v^2}{r} \ {\rm d}v_{\hn}.
		\end{align*}   
		
		Now expanding $v$ in spherical harmonics $v(x):=v(r,\sigma)=\sum_{n=0}^{\infty}d_n(r)P_n(\sigma)$, we obtain
		
		\begin{align*}
		\int_{\hn} \frac{\rgt}{r^{\alpha}} \ {\rm d}v_{\hn}&=\sum_{n=0}^{\infty}\bigg[\int_{0}^{\infty}{d'_n}^2 \ {\rm d}r+\frac{(N-1)^2}{4}\int_{0}^{\infty}(\coth r)^2 \ d_n^2 \ {\rm d}r+\frac{\alpha^2}{4}\int_{0}^{\infty}\frac{d_n^2}{r^2} \ {\rm d}r\\ &-(N-1)\int_{0}^{\infty} (\coth r)\ d'_n d_n  \ {\rm d}r
		+\alpha\int_{0}^{\infty}\frac{d'_n d_n}{r} \ {\rm d}r-\frac{\alpha(N-1)}{2}\int_{0}^{\infty}(\coth r)\frac{d_n^2}{r} \ {\rm d}r\bigg]\\
		&
		=\sum_{n=0}^{\infty}\bigg[\int_{0}^{\infty}{d'_n}^2 \ {\rm d}r+\frac{(N-1)^2}{4}\int_{0}^{\infty}(\coth r)^2 \ d_n^2 \ {\rm d}r+\frac{\alpha^2}{4}\int_{0}^{\infty}\frac{d_n^2}{r^2} \ {\rm d}r\\ &-\frac{(N-1)}{2}\int_{0}^{\infty} \frac{d_n^2}{(\sinh r)^2}  \ {\rm d}r
		+\frac{\alpha}{2}\int_{0}^{\infty}\frac{d_n^2}{r^2} \ {\rm d}r-\frac{\alpha(N-1)}{2}\int_{0}^{\infty}(\coth r)\frac{d_n^2}{r} \ {\rm d}r\bigg].
		\end{align*}
		
		Observing that
		\begin{align*}
		&\frac{(N-1)^2}{4}\int_{0}^{\infty}(\coth r)^2 \ d_n^2 \ {\rm d}r -\frac{(N-1)}{2}\int_{0}^{\infty} \frac{d_n^2}{(\sinh r)^2}  \ {\rm d}r\\&=\frac{(N-1)^2}{4}\int_{0}^{\infty}  d_n^2 \ {\rm d}r + \frac{(N-1)(N-3)}{4}\int_{0}^{\infty} \frac{d_n^2}{(\sinh r)^2}  \ {\rm d}r, 
		\end{align*}
		and using $1$-dimensional Hardy inequality and $(\coth r)\geq 1/r$, we infer
		\begin{align*}
		&\int_{\hn} \frac{\rgt}{r^{\alpha}} \ {\rm d}v_{\hn}  \geq
		\sum_{n=0}^{\infty}\bigg[\frac{(\alpha+1)^2}{4}\int_{0}^{\infty}\frac{d_n^2}{r^2}\ {\rm d}r +\frac{(N-1)^2}{4}\int_{0}^{\infty}
		d_n^2 \ {\rm d}r\notag \\ &+ \frac{(N-1)(N-3)}{4}\int_{0}^{\infty} \frac{d_n^2}{(\sinh r)^2} \ {\rm d}r-\frac{\alpha(N-1)}{2}\int_{0}^{\infty}(\coth r)\frac{d_n^2}{r}  \ {\rm d}r\bigg]\\ 
		&=\sum_{n=0}^{\infty}\bigg[\frac{(\alpha+1)^2}{4}\int_{0}^{\infty}\frac{d_n^2}{r^{2}}\ {\rm d}r +\bigg[\frac{(N-1)^2}{4}-\frac{(N-1)(N-3)}{4}\bigg]\int_{0}^{\infty} d_n^2 \ {\rm d}r\\
		&+ \frac{(N-1)(N-3)}{4}\int_{0}^{\infty} (\coth r)^2d_n^2 \ {\rm d}r -\frac{\alpha(N-1)}{2}\int_{0}^{\infty}(\coth r)\frac{d_n^2}{r} \ {\rm d}r\bigg]\\
		&\geq \sum_{n=0}^{\infty}\bigg[ \frac{(\alpha+1)^2}{4}\int_{0}^{\infty}\frac{d_n^2}{r^{2}} {\rm d}r+\frac{(N-1)}{2}\int_{0}^{\infty} d_n^2  {\rm d}r+\bigg[\frac{(N-1)(N-3)}{4}-\frac{\alpha(N-1)}{2}\bigg]\int_{0}^{\infty}(\coth r)\frac{d_n^2}{r}  {\rm d}r\bigg]\\
		&=\sum_{n=0}^{\infty}\bigg[\frac{(N-2-\alpha)^2}{4} \int_{0}^{\infty}\frac{d_n^2}{r^{2}}\ {\rm d}r+\frac{(N-1)}{2}\int_{0}^{\infty} d_n^2 \ {\rm d}r+\frac{(N-1)(N-3-2\alpha)}{4}\int_{0}^{\infty} g(r)\ d_n^2 \ {\rm d}r\bigg].
		\end{align*}
		
		Finally writing all the above terms w.r.t. $u$ we establish our desired Theorem \ref{th_1}. Optimality of the constant $\frac{(N-2-\alpha)^2}{4}$ was already established in \cite[Theorem 3.1]{VHN}.
	\end{proof}

\begin{remark}
	The coefficient in front of the last term in \eqref{eq_th_1} is negative whenever $N-3<2\alpha$. Note that $g(r)\leq 1/3$ for every $r> 0$, we deduce 
	\begin{align*}
	\frac{(N-1)}{2}+\frac{(N-1)(N-3-2\alpha)}{12}=\frac{(N-1)(N+3-2\alpha)}{12}>0
	\end{align*}
	for $N+3>2\alpha$. Hence, the initial restriction of dimension in \eqref{eq_th_1} is justified. Also note that exploiting Gauss's Lemma in \eqref{eq_th_1} we can obtain different version of weighted Hardy inequality. Another implication of \eqref{eq_th_1} is an  immediate improvement of \cite[Theorem 3.1]{VHN} for the case $p=2$.
\end{remark}

By granting, $N-3\geq 2\alpha$ in \eqref{eq_th_1}, one has the following corollary:
\begin{corollary}\label{cor_1}
	Let $0\leq 2\alpha \leq N-3$. Then, for all $u\in\cchnm$, there holds 
	\begin{align}\label{eq_cor_1}
	& \int_{\hn} \frac{\rgt}{r^{\alpha}} \ {\rm d}v_{\hn} \, \geq  \,
	\frac{(N-2-\alpha)^2}{4} \int_{\hn} \frac{u^2}{r^{\alpha+2}} \ {\rm d}v_{\hn} +\frac{(N-1)}{2}\int_{\hn} \frac{u^2}{r^{\alpha}} \ {\rm d}v_{\hn}.
	\end{align}
Furthermore, the constant $\frac{(N-2-\alpha)^2}{4}$ is sharp in the obvious sense.
\end{corollary}

Now we will develop weighted Rellich type inequality with Hardy type remainder terms which is an analogous result of \cite[Theorem 5.2]{BGGP}. Before going into detail first recall another important lemma.

\begin{lemma}\label{lemma_2}
	For all $u\in \cchn$, there holds $\lr (u^2)=2u(\lr u) + 2\rgt$.
\end{lemma}
\begin{proof}
	This follows from it's own definition and by simple calculation 
	\begin{align*}
	\lr (u^2)&=\frac{\partial^2 u^2}{\partial r^2}+(N-1)(\coth r) \frac{\partial u^2}{\partial r}=2\frac{\partial}{\partial r}\big(u\frac{\partial u}{\partial r}\big) +2(N-1)(\coth r) \frac{\partial u}{\partial r} u\\
	&=2u\frac{\partial^2 u}{\partial r^2}+2\big(\frac{\partial u}{\partial r}\big)^2+2u(N-1)(\coth r) \frac{\partial u}{\partial r}=2u(\lr u) + 2\rgt.
	\end{align*}\end{proof}

Exploiting Lemma \ref{lemma_2} and Theorem \ref{eq_th_1}, we state a weighted Rellich inequality.
\begin{theorem}\label{th_2}
    Let $\alpha$ be a positive number and $N>\max\{\alpha+2,2\alpha-3\}$. For all $u\in\cchnm$, there holds 
	\begin{align}\label{eq_th_2}
	\int_{\hn} \frac{|\lr u|^2}{r^{\alpha-2}}\dv& \geq \frac{(N-2-\alpha)^2(N-2+\alpha)^2}{16}\int_{\hn} \frac{u^2}{r^{\alpha +2}}\dv \\& \notag + \frac{(N-2-\alpha)(N-2+\alpha)(N-1)}{4}\int_{\hn} \frac{u^2}{r^{\alpha}}\dv \\ & + \notag  \frac{(N-1)(N-3-2\alpha)(N-2-\alpha)(N-2+\alpha)}{8} \int_{\hn} g(r)\frac{u^2}{r^{\alpha}} \, {\rm d}v_{\hn},
	\end{align}
	where $g(r)$ is as defined in \eqref{eq_th_1}. Moreover, the constant $\frac{(N-2-\alpha)^2(N-2+\alpha)^2}{16}$ is sharp in the obvious sense.
\end{theorem}
\begin{proof}
	This proof mainly relies on the inequality \eqref{eq_th_1}. Notice that, whenever $\alpha>0$, there holds
	\begin{align*}
	-\lr \frac{1}{r^\alpha} = - \Delta_{\hn} \frac{1}{r^\alpha}\geq \frac{\alpha(N-2-\alpha)}{r^{\alpha+2}} \text{ for } r>0.
	\end{align*}
	
	First we multiply above by $u^2$ and after that performing by parts formula, Lemma \ref{lemma_2} and Young's inequalty with $\epsilon>0$, we deduce
	\begin{align*}
	 \int_{\hn} \frac{|\lr u|^2}{r^{\alpha-2}}\dv& \geq 2\epsilon\int_{\hn} \frac{|\gr u|^2}{r^\alpha}\dv +  \big[\epsilon\alpha(N-2-\alpha)-\epsilon^2\big]\int_{\hn} \frac{u^2}{r^{\alpha+2}}\dv\\&   \geq \big[\epsilon\alpha(N-2-\alpha)-\epsilon^2+\epsilon\frac{(N-2-\alpha
		)^2}{2}\big]\int_{\hn} \frac{u^2}{r^{\alpha+2}}\dv+\epsilon(N-1)\int_{\hn} \frac{u^2}{r^{\alpha}}\dv \\& + \epsilon \frac{(N-1)(N-3-2\alpha)}{2} \int_{\hn} g(r)\frac{u^2}{r^{\alpha}} \, {\rm d}v_{\hn}.
	\end{align*}
	
	Now the coefficient in front of $\int_{\hn} u^2/r^{\alpha+2}\dv$ will be maximum when $\epsilon = \frac{(N-2-\alpha)(N-2+\alpha)}{4}$ and substituting this we obtain our required result. Optimality issue of the constant $\frac{(N-2-\alpha)^2(N-2+\alpha)^2}{16}$ was already tackled in \cite[Theorem 4.3]{VHN}.
\end{proof}
\begin{remark}
Exploiting \cite[Lemma 6.1]{EGR} in \eqref{eq_th_2} we can deduce another version of weighted Rellich inequality with Hardy type remainder terms. Also it is important to notice that, \eqref{th_2} gives one more version of immediate improvement of \cite[Theorem 4.3]{VHN} for the case $p=2$.  
\end{remark}

Collecting the conditions in Theorem \ref{th_2} and $2 \alpha \leq N-3 $, one has the following corollary:
	\begin{corollary}\label{cor_2}
		Let $0\leq 2 \alpha\leq N-3 $. Then for all $u\in\cchnm$, there holds 
		\begin{align}\label{eq_cor_2}
		 \int_{\hn} \frac{|\lr u|^2}{r^{\alpha-2}} \ {\rm d}v_{\hn} & \geq \frac{(N-2-\alpha)^2(N-2+\alpha)^2}{16}\int_{\hn} \frac{u^2}{r^{\alpha +2}}\ {\rm d}v_{\hn} \\& \notag + \frac{(N-2-\alpha)(N-2+\alpha)(N-1)}{4}\int_{\hn} \frac{u^2}{r^{\alpha}} \ {\rm d}v_{\hn}.
		\end{align}
		Moreover, the constant $\frac{(N-2-\alpha)^2(N-2+\alpha)^2}{16}$ is sharp in the obvious sense.
	\end{corollary}

In the rest of the part, we will construct more weighted Hardy and Rellich type inequalities in terms of radial derivatives. Most of the ideas are taken from \cite{YSK}. It is worth mentioning that, here we will only discuss the results for the case $p=2$ but one can verify that, same things hold true in the case of $L^p$ Hardy inequality on $\hn$, with $p\geq2$. First, we describe an important lemma below.

\begin{lemma}\label{lemma_3}
	Let $N\geq 3$. For all $u\in\cchnm$, there holds
	\begin{align}\label{eq_lemma_3}
	\int_{\hn}r^{2-N}|u|^2\dv\leq 4 \int_{\hn}r^{2-N}|\gr u|^2\dv.
	\end{align}
\end{lemma}
\begin{proof}
	First observe that 
	\begin{align*}
	\frac{[r^{2-N}(\sinh r)^{N-1}]^{'}}{[r^{2-N}(\sinh r)^{N-1}]}=\frac{1}{r} +(N-1)\big[\coth r -\frac{1}{r}\big]\geq 1.
	\end{align*}
	
	Indeed, the above inequality holds true. It is easy to see the above inequality is equivalent to the following 
	$$(N-1)(r\coth r -1)\geq (r-1).$$
	First consider the case, $r \geq 1,$ then the above holds true follows from the fact that $N \geq 3$ and $\coth r > 1,$ $r > 0.$
	In the remaining case, i.e., for $0<r<1$, inequality holds true follows from the fact that $[\coth r -\frac{1}{r}]\geq 0$ and $\frac{1}{r} > 1.$ Then, exploiting by parts formula and H\"older inequality into above, we derive
	\begin{align*}
	& \int_{\hn}r^{2-N}|u|^2\dv=\int_{\mathbb{S}^{N-1}}\int_{0}^{\infty}r^{2-N}(\sinh r)^{N-1} |u|^2 \ {\rm d}r \ {\rm d}\sigma \\& \leq \int_{\mathbb{S}^{N-1}}\int_{0}^{\infty}[r^{2-N}(\sinh r)^{N-1}]^{'} |u|^2 \ {\rm d}r \ {\rm d}\sigma =-2\int_{\mathbb{S}^{N-1}}\int_{0}^{\infty}r^{2-N}(\sinh r)^{N-1} u \frac{\partial u}{\partial r}  \ {\rm d}r \ {\rm d}\sigma\\& \leq 2 \bigg( \int_{\mathbb{S}^{N-1}}\int_{0}^{\infty}r^{2-N}(\sinh r)^{N-1} |u|^2  \ {\rm d}r \ {\rm d}\sigma\bigg)^{1/2} \bigg( \int_{\mathbb{S}^{N-1}}\int_{0}^{\infty}r^{2-N}(\sinh r)^{N-1} |\gr u|^2  \ {\rm d}r \ {\rm d}\sigma\bigg)^{1/2}.
	\end{align*}	
	
	Finally shifting in the original coordinate we get the desired result.
\end{proof}

    Let us define the quantity 
	\begin{align*}
	\mu_r(\hn)=\underset{u\in\cchnm\setminus\{0\}}{\inf}\frac{\int_{\hn}r^{2-N}|\gr u|^2\dv}{\int_{\hn}r^{2-N}|u|^2\dv}.
	\end{align*}
	
	In turn of Lemma \ref{lemma_3}, we deduce $\mu_r(\hn)\geq1/4$. In fact better estimate of $\mu_r(\hn)$ holds true.
	
	\begin{lemma}\label{lemma_4}
		Let $N\geq 3$. Then there holds $\mu_r(\hn)\geq \frac{N-1}{4}$.
	\end{lemma}
		\begin{proof}
		Start with the function $u=\big(2\cosh^2 (r/2))^{(1-N)/2}v$. After going along with the exactly same estimate in \cite[Theorem 5.2]{YSK}, we will deduce the result. Taking the advantage of radial function $\zeta(r)=\big(2\cosh^2 (r/2))^{(1-N)/2}$ and exploiting Lemma \ref{lemma_1}, we will arrive at the same conclusion.
        \end{proof}
	
	Now we are ready to establish the analogous version of \cite[Theorem 4.2]{YSK} and due to Gauss's Lemma, the following theorem comes out as a stronger version of it.
	
		\begin{theorem}\label{th_3}
		Let $0\leq \alpha <N-2$ with $N\geq 3$. Then for all $u\in\cchnm$, there holds 
		\begin{align}\label{eq_th_3}
		\int_{\hn} \frac{|\gr u|^2}{r^\alpha}\dv\geq \frac{(N-2-\alpha)^2}{4}\int_{\hn} \frac{u^2}{r^{\alpha +2}}\dv + \frac{(N-1)}{4}\int_{\hn} \frac{u^2}{r^{\alpha}}\dv,
		\end{align}
		where the constant $\frac{(N-2-\alpha)^2}{4}$ is sharp in the obvious sense.
	\end{theorem}
	\begin{proof}
		Start with the substitution $u=r^{(2+\alpha-N)/2}v$ and from a simple calculation, we have that
		\begin{align*}
		|\gr u|^2=\bigg(\frac{2+\alpha-N}{2}\bigg)^2 \bigg(\frac{u}{r}\bigg)^2 + r^{2+\alpha-N}|\gr v|^2  + 2\bigg(\frac{2+\alpha-N}{2}\bigg) r^{1+\alpha-N} v \frac{\partial v}{\partial r}.
		\end{align*}
		
		Before performing integration, first multiply above by $1/r^{\alpha}$ and we obtain
		\begin{align*}
		\int_{\hn} \frac{|\gr u|^2}{r^\alpha} \ {\rm d}v_{\hn}& =\bigg(\frac{N-2-\alpha}{2}\bigg)^2\int_{\hn} \frac{u^2}{r^{\alpha +2}} \ {\rm d}v_{\hn} + \int_{\hn} r^{2-N}|\gr v|^2 \ {\rm d}v_{\hn}\\ & -(N-\alpha-2)\int_{\hn} r^{1-N} v \frac{\partial v}{\partial r} \ {\rm d}v_{\hn}.
		\end{align*}
		
		Transfering into polar coordinate and using by parts rule, we deduce
		\begin{align*}
		-(N-\alpha-2)\int_{\hn} r^{1-N} v \frac{\partial v}{\partial r} \dv=\frac{(N-\alpha-2)(N-1)}{2}\int_{\hn} r^{1-N}v^2\bigg[\coth r -\frac{1}{r} \bigg]\dv.
		\end{align*}
		
		Exploiting taylor series expansion of $\cosh r$ and $\sinh r$ near origin, for $0<r\leq 1$, we deduce 
		\begin{align*}
		(\coth r-\frac{1}{r})& =\frac{1}{r\sinh r}\big(r\cosh r - \sinh r\big) \\ & =\frac{1}{r\sinh r} \bigg(r\sum_{n=0}^{\infty}\frac{r^{2n}}{(2n)!}-\sum_{n=0}^{\infty}\frac{r^{2n+1}}{(2n+1)!}\bigg)\\& \geq \frac{1}{r\sinh r}\cdot \frac{r^3}{3}\\ & \geq\frac{r}{3\sinh 1}.	
		\end{align*}
	The last inequality follows from the fact that, the function $g(r):=r\sinh 1-\sinh r$ is in $C^2([0,1])$ and concave in $[0,1]$ with zero on the boundary and hence $g(r)\geq 0$ whenever $0<r\leq 1$. By the above estimate along with the Lemma \ref{lemma_4} and getting back into the form of $u$, we derive
		\begin{align*}
		& \int_{\hn} \frac{|\gr u|^2}{r^\alpha}\dv \geq\bigg(\frac{N-2-\alpha}{2}\bigg)^2\int_{\hn} \frac{u^2}{r^{\alpha +2}}\dv \\&\notag +\frac{(N-1)}{4}\int_{\hn} \frac{u^2}{r^\alpha}\dv+\frac{(N-\alpha-2)(N-1)}{6\sinh 1}\int_{B(o;1)}\frac{u^2}{r^\alpha}\dv.
		\end{align*}
		
		In the end, non-negativity of the last term immdiately gives \eqref{eq_th_2}.
	\end{proof}

Taking Theorem \ref{th_3} as weighted Hardy inequality and adopting the similar technique exploited in Theorem \ref{th_2}, one has the following version of weighted Rellich inequality.
\begin{corollary}\label{cor_3}
	Let $0\leq \alpha <N-2$ with $N\geq 3$. Then for all $u\in\cchnm$ there holds 
	\begin{align}\label{eq_cor_3}
	\int_{\hn} \frac{|\lr u|^2}{r^{\alpha-2}} \ {\rm d}v_{\hn} & \geq \frac{(N-2-\alpha)^2(N-2+\alpha)^2}{16}\int_{\hn} \frac{u^2}{r^{\alpha +2}} \ {\rm d}v_{\hn} \\& \notag + \frac{(N-2-\alpha)(N-2+\alpha)(N-1)}{8}\int_{\hn} \frac{u^2}{r^{\alpha}} \ {\rm d}v_{\hn}.
	\end{align}
	Moreover, the constant $\frac{(N-2-\alpha)^2(N-2+\alpha)^2}{16}$ is sharp in the obvious sense.
\end{corollary}

Observing into both the weighted Rellich inequalities  \eqref{eq_cor_2} and \eqref{eq_cor_3}, one can wonder, whether more better improvement possible near origin, precisely can we add one more Hardy type remainder term namely, $\int_{\hn} u^2/r^{\alpha-2}\dv$. To give affirmative answer of this question, first we develop the following lemma.
\begin{lemma}\label{lemma_5}
	Let $-2\leq \alpha <N-4$ and $u\in \cchnm.$ Then there holds 
	\begin{align*}
	&\int_{\hn}r^{-\alpha}|\lr u +\frac{(N+\alpha)(N-\alpha-4)}{4}\frac{u}{r^2}|^2\dv  \leq \int_{\hn} \frac{\rlt}{r^\alpha}\dv \\&-\frac{(N+\alpha)(N-\alpha-4)}{2}\int_{\hn} \frac{|\gr u|^2}{r^{\alpha+2}} \, {\rm d}v_{\hn} +\frac{(N+\alpha)(N-3\alpha-8)(N-\alpha-4)^2}{16}\int_{\hn}\frac{u^2}{r^{\alpha+4}}\dv.
	\end{align*}
\end{lemma}
\begin{proof}
	Exploiting $(\coth r)\geq 1/r$, we deduce
	\begin{align*}
	\lr(r^{-\alpha-2})&\leq(\alpha+2)(\alpha+4-N)r^{-\alpha-4}.
	\end{align*}
	
	Applying by parts formula and Lemma \ref{lemma_2}, in the above inequality, we infer
	\begin{align}\label{use_lemma_1}
	\int_{\hn} \frac{u\lr u}{r^{\alpha+2}} \dv\leq -\frac{(\alpha+2)(N-\alpha-4)}{2} \int_{\hn} \frac{u^2}{r^{\alpha+4}}\dv - \int_{\hn} \frac{\rgt}{r^{\alpha+2}}\dv.
	\end{align}
	
	Finally, the conclusion comes by noting that
	\begin{align*}
	&\int_{\hn}r^{-\alpha}|\lr u +\frac{(N+\alpha)(N-\alpha-4)}{4}\frac{u}{r^2}|^2\dv  = \int_{\hn} \frac{\rlt}{r^\alpha}\dv \\ &+ \frac{(N+\alpha)^2(N-\alpha-4)^2}{16}\int_{\hn}\frac{u^2}{r^{\alpha+4}}\dv+\frac{(N+\alpha)(N-\alpha-4)}{2}\int_{\hn} \frac{u\lr u}{r^{\alpha+2}}\dv.
	\end{align*}
\end{proof}
 
 Now using Lemma \ref{lemma_5} and weighted Hardy inequality \eqref{eq_th_3}, we obtain the following weighted Rellich type inequality. Also note that, exploiting \cite[Lemma 6.1]{EGR}, this version will become stronger than \cite[Theorem 4.4]{YSK} and will give a quick improvement of \cite[Theorem 4.3]{VHN}, for the case $p=2$.
 \begin{theorem}\label{th_4}
 	Let $0\leq \alpha <N-4$. Then for all $u\in\cchnm$, there holds 
 	\begin{align}\label{eq_th_4}
 	& \int_{\hn} \frac{|\lr u|^2}{r^\alpha}\dv \geq \frac{(N+\alpha)^2(N-4-\alpha
 		)^2}{16}\int_{\hn} \frac{u^2}{r^{\alpha+4}} \, {\rm d}v_{\hn}\\ & \notag
 	+ \frac{(N-1)(N-2-\alpha)(N-2+\alpha
 		)}{8}\int_{\hn} \frac{u^2}{r^{\alpha+2}} \, {\rm d}v_{\hn}+\frac{(N-1)^2}{16}\int_{\hn} \frac{u^2}{r^{\alpha}} \, {\rm d}v_{\hn}.
 	\end{align}
 	Moreover, the constant $\frac{(N+\alpha)^2(N-4-\alpha
 		)^2}{16}$ is sharp in the obvious sense.
 \end{theorem}
 \begin{proof}
 	Replacing the index $\alpha$ by $\alpha-2$ in \eqref{use_lemma_1} and then substituting  \eqref{eq_th_3} into it, we deduce
 	\begin{align*}
 	(\alpha+1)\int_{\hn} \frac{u^2}{r^{\alpha+2}}\dv +\frac{(N-1)}{4}\int_{\hn} \frac{u^2}{r^{\alpha}}\dv \leq \int_{\hn}\frac{u}{r^\alpha}\bigg[-\lr u-\frac{(N+\alpha)(N-\alpha-4)}{4}\frac{u}{r^2}\bigg]\dv.
 	\end{align*}
 	
 	Now we estimate the last term by Young's inequality with $\epsilon>0$ and taking $a=\big|\frac{u}{r^{\alpha/2}}\big|$ and $b=\big|r^{-\alpha/2}\big[-\lr u-\frac{(N+\alpha)(N-\alpha-4)}{4}\frac{u}{r^2}\big]\big|$, we obtain
 	\begin{align*}
 	& 2\epsilon(\alpha+1)\int_{\hn}\frac{u^2}{r^{\alpha+2}} \  {\rm d}v_{\hn}+\frac{\epsilon(N-1-2\epsilon)}{2}\int_{\hn} \frac{u^2}{r^{\alpha}} \ {\rm d}v_{\hn} \\ & \leq  \int_{\hn}r^{-\alpha}|\lr u +\frac{(N+\alpha)(N-\alpha-4)}{4}\frac{u}{r^2}|^2 \ {\rm d}v_{\hn}.
 	\end{align*}
 	
 	Next we exploit the information that, function $f(\epsilon)=\epsilon(N-1-2\epsilon)/2$ attains maximum when $\epsilon=(N-1)/4$. Finally, applying Lemma \ref{lemma_5} and Theorem \ref{th_3}, in the form of changed index $\alpha$ by $\alpha+2$, we achieve our desired result. 
 \end{proof}
 
  We establish Theorem \ref{th_4}, using Theorem \ref{th_3} and Lemma \ref{lemma_5}. On the other hand, in a similar way, we can deduce Corollary \ref{cor_4} using Corollary \ref{cor_1} instead of Theorem \ref{th_3} in the proof of Theorem \ref{th_4}.
 \begin{corollary}\label{cor_4}
 	Let $0\leq  2\alpha\leq N-7$. Then for all $u\in \cchnm$, there holds 
 	\begin{align}\label{eq_cor_4}
 	&\int_{\hn} \frac{\rlt}{r^\alpha}\dv \geq\frac{(N+\alpha)^2(N-4-\alpha)^2}{16}\int_{\hn}\frac{u^2}{r^{\alpha+4}}\dv\\ & \notag + \frac{(N-1)(N-2-\alpha)(N-2+\alpha)}{4}\int_{\hn} \frac{u^2}{r^{\alpha+2}} \, {\rm d}v_{\hn} \notag
 	+\frac{(N-1)^2}{4}\int_{\hn}\frac{u^2}{r^{\alpha}}\dv.
 	\end{align}
 	Moreover, the constant $\frac{(N+\alpha)^2(N-4-\alpha)^2}{16}$ is sharp in the obvious sense.
 \end{corollary}
\begin{remark}
	If we compare both the weighted Hardy inequalities in Corollary \ref{cor_1} and Theorem \ref{th_3}, one can observe that coefficient in front of $\int_{\hn}u^2/r^\alpha\dv$ in the equation \eqref{eq_cor_1} is better than \eqref{eq_th_3}. But it's also important to notice that, \eqref{eq_cor_1} demands larger dimension restriction than \eqref{eq_th_3}. Analogous observation also holds true for the case of weighted Rellich inequalities in Corollary \ref{cor_4} and Theorem \ref{th_4}. Moreover, for both the cases we are getting one instance, where $\mu_r(\hn)>(N-1)/4$ is possible.
\end{remark}

Iterating inequality \eqref{eq_th_4}, we obtain the following improved weighted Rellich inequality on higher order radial derivation on $\hn$ and this result will be used many times in the last part of the article.
\begin{lemma}\label{lemma_6}
	Let $\beta$ be a positive integer, which satisfy $0\leq \alpha< N-4\beta$. Then there exist positive constants $\Xi_{\alpha,\beta}^j$, for $j=0 \text{ to }2\beta$, such that for all $u\in\cchnm$, there holds 
	\begin{align}\label{eq_lemma_6}
	\int_{\hn} \frac{(\lr^{\beta} u)^2}{r^\alpha}\dv & \geq \sum_{j=0}^{2\beta}\Xi_{\alpha,\beta}^j\int_{\hn}\frac{u^2}{r^{\alpha+4\beta-2j}}\dv.
	\end{align}
	Moreover, the coefficient corresponding to the leading terms namely $\Xi_{\alpha,\beta}^0$ and $\Xi_{\alpha,\beta}^{2\beta}$, for $r\rightarrow 0$ and $r\rightarrow \infty$ respectively, can be explicitly given by as follows
	\begin{equation*}
	\Xi_{\alpha,\beta}^0=\prod_{j=0}^{\beta-1}\frac{(N+(\alpha+4j))^2(N-(\alpha+4j)-4)^2}{16} \text{ and } \  \Xi_{\alpha,\beta}^{2\beta}=\bigg(\frac{N-1}{4}\bigg)^{2\beta} \text{ for }\beta\geq 1 \text{ and }\alpha\geq 0.
	\end{equation*}
\end{lemma}

Finally after iterating \eqref{eq_cor_4}, we deduce the following result but with a different initial condition.
\begin{lemma}\label{lemma_7}
	Let $\beta$ be a positive integer, which satisfy $0\leq 2\alpha\leq N-8\beta+1$. Then there exist positive constants $\zeta_{\alpha,\beta}^j$, for $j=0 \text{ to }2\beta$, such that for all $u\in\cchnm$, there holds 
	\begin{align}\label{eq_lemma_7}
	\int_{\hn} \frac{(\lr^{\beta} u)^2}{r^\alpha}\dv & \geq \sum_{j=0}^{2\beta}\zeta_{\alpha,\beta}^j\int_{\hn}\frac{u^2}{r^{\alpha+4\beta-2j}}\dv,
	\end{align}
	where $\zeta_{\alpha,\beta}^0=\Xi_{\alpha,\beta}^0$ and $\zeta_{\alpha,\beta}^{2\beta}=4^\beta \ \Xi_{\alpha,\beta}^{2\beta}, \text{ for }\beta\geq 1 \text{ and }\alpha\geq 0$.
\end{lemma}
In the rest of the article for notational convention we will be assuming 
$\Xi_{n,0}^{0}=1$ and $\zeta_{n,0}^{0}=1$, for every integer $n$. Also we will assume $\sum_{j=m}^{n}=0$ and $\prod_{j=m}^{n}=1$, whenever integers satisfy $n<m$.

\medskip

\section{Improvement of Higher order radial Poincar\'e Inequality}\label{sect_improve}

This section is devoted to the proof of \eqref{imp_r_high_poin}. In the same spirit to explore further in l.h.s. of \eqref{r21d}, exploiting  \eqref{eq_th_3} for the case of $\alpha=2$ into it, we deduce with a different constant than \cite[Theorem 2.1]{EG} that, for all $u\in\cchnm$ and $N\geq 5$, there holds,
\begin{align}\label{r21}
 \int_{\hn} |\Delta_{r} u|^2  \ {\rm d}v_{\hn} -  \left( \frac{N-1}{2} \right)^{2} \int_{\hn} |\nabla_{r} u|^2  \ {\rm d}v_{\hn}
  \geq \frac{(N-4)^2}{16} \int_{\hn} \frac{u^2}{r^{4}} \ {\rm d}v_{\hn}+\frac{(N-1)}{16}\int_{\hn} \frac{u^2}{r^{2}} \ {\rm d}v_{\hn}.
\end{align}

Furthermore, using \eqref{r10} in \eqref{r21}, we obtain for all $u\in\cchnm$ with $N\geq 5$ there holds 
\begin{align}\label{r20}
\int_{\hn} |\Delta_{r} u|^2\ {\rm d}v_{\hn} -  \left( \frac{N-1}{2} \right)^{4} \int_{\hn} u^2 \ {\rm d}v_{\hn}
\geq \frac{(N-4)^2}{16} \int_{\hn} \frac{u^2}{r^{4}} \ {\rm d}v_{\hn}+\frac{N(N-1)}{16}\int_{\hn} \frac{u^2}{r^{2}} \ {\rm d}v_{\hn}.
\end{align}

Indeed, all these lower order improvements can be lifted into the general higher order indices scenario. In particular, applying these lower order indices results and induction we will approach towards the development of the result \eqref{imp_r_high_poin}. In the coming part, we will mainly rely on the Lemma \ref{lemma_6} and Lemma \ref{lemma_7}. We divide this section into a couple of subsections to cover up all the possible higher order indices $k,l$ and side by side we will explicitly calculate the coefficients corresponding to the asymptotic terms $r\rightarrow 0$ and $r\rightarrow\infty$ also. 

\subsection{General integer $k$ and $l=0$} 
This part is divided into two steps based on the situation $k$ is odd or even. First we state the results and after that we will give the details of the proof.
\begin{theorem}
	Let $k$ be a positive integer and $N>2k$. Then there exist $k$ positive constants $C_{k,0}^i$ such that the following inequality holds 
	\begin{equation}\label{even_and_odd}
	\int_{\hn} |\gr^k u|^2\dv - \bigg(\frac{N-1}{2}\bigg)^{2k}\int_{\hn}u^2\dv
	\geq \sum_{i=1}^{k}C_{k,0}^{i}\int_{\hn}\frac{u^2}{r^{2i}}\dv,
	\end{equation}
	for all $u\in\cchnm$. Moreover, the leading terms are explicitly given by
	\begin{equation*}
	C_{k,0}^k=
	\begin{dcases}
	\bigg(\frac{N-4}{2^{2m}}\bigg)^2\ \prod_{j=1}^{m-1}(N+4j)^2(N-4j-4)^2 & \text{ if } k=2m, \\
	 	\frac{1}{2^{4m+2}}\prod_{j=1}^{m}(N+4j-2)^2(N-4j-2)^2  & \text{ if } k=2m+1, \\
	 	\end{dcases}
	\end{equation*} and
	\begin{equation*}
	C_{k,0}^1=
	\begin{dcases}
	\frac{N(N-1)}{2^{4m}}\sum_{j=1}^{m}(N-1)^{4m-2j-2} & \text{ if } k=2m, \\
\frac{N(N-1)}{2^{4m+2}}\sum_{j=1}^{m}(N-1)^{2m+2j-2} +\frac{(N-1)^{2m}}{2^{4m+2}} & \text{ if } k=2m+1. \\
	\end{dcases}
	\end{equation*}
\end{theorem}

\begin{proof}
	Suppose $k=2m$ even, we will apply induction $m$. For the basic step we already have the result in \eqref{r20}. Now assume it holds true for the case $k=2m-2\geq 2$, which describes that for $N>4m-4$, there holds
	\begin{align*}
	& \int_{\hn} (\lr^{m-1} u)^2\dv - \bigg(\frac{N-1}{2}\bigg)^{4m-4}\int_{\hn}u^2\dv
	\geq \frac{N(N-1)}{2^{4m-4}}\sum_{j=1}^{m-1}(N-1)^{4m-2j-6}\int_{\hn}\frac{u^2}{r^{2}}\ {\rm d}v_{\hn} \\ & + \sum_{i=2}^{2m-3}C_{2m-2,0}^{i}\int_{\hn}\frac{u^2}{r^{2i}}\dv +\bigg(\frac{N-4}{2^{2m-2}}\bigg)^2\ \prod_{j=1}^{m-2}(N+4j)^2(N-4j-4)^2\int_{\hn}\frac{u^2}{r^{4m-4}}\ {\rm d}v_{\hn}.
	\end{align*}
	
	Next we will establish the inductive step and so starting with $N>4m$, exploiting induction hypothesis above, \eqref{r20} and \eqref{eq_lemma_6}, we deduce
	\begin{align*}
	& \int_{\hn}|\lr^{m} u|^2 \ {\rm d}v_{\hn} =\int_{\hn}|\lr^{m-1}(\lr u)|^2 \ {\rm d}v_{\hn} \geq \bigg(\frac{N-1}{2}\bigg)^{4m-4}\int_{\hn}(\lr u)^2 \ {\rm d}v_{\hn}\\ & +\sum_{i=2}^{2m-3}C_{2m-2,0}^{i}\int_{\hn}\frac{(\lr u)^2}{r^{2i}} \ {\rm d}v_{\hn}+  \frac{N(N-1)}{2^{4m-4}}\sum_{j=1}^{m-1}(N-1)^{4m-2j-6}\int_{\hn}\frac{(\lr u)^2}{r^{2}} \ {\rm d}v_{\hn}\\& +\bigg(\frac{N-4}{2^{2m-2}}\bigg)^2\ \prod_{j=1}^{m-2}(N+4j)^2(N-4j-4)^2\int_{\hn}\frac{(\lr u)^2}{r^{4m-4}} \ {\rm d}v_{\hn}\\ & \geq \bigg(\frac{N-1}{2}\bigg)^{4m-4}\bigg[\left( \frac{N-1}{2} \right)^{4} \int_{\hn} u^2 \ {\rm d}v_{\hn} 
	+ \frac{(N-4)^2}{16} \int_{\hn} \frac{u^2}{r^{4}} \ {\rm d}v_{\hn}+\frac{N(N-1)}{16}\int_{\hn} \frac{u^2}{r^{2}} \ {\rm d}v_{\hn}\bigg]\\ & + \sum_{i=2}^{2m-3}C_{2m-2,0}^{i}\int_{\hn}\frac{(\lr u)^2}{r^{2i}}\dv+ \frac{N(N-1)}{2^{4m-4}}\sum_{j=1}^{m-1}(N-1)^{4m-2j-6} \bigg[\sum_{\gamma=0}^{2}\Xi_{2,1}^{\gamma}\int_{\hn}\frac{u^2}{r^{6-2\gamma}}\dv\bigg]
	\\ &+ \bigg(\frac{N-4}{2^{2m-2}}\bigg)^2\ \prod_{j=1}^{m-2}(N+4j)^2(N-4j-4)^2\bigg[\sum_{\gamma=0}^{2}\Xi_{4m-4,1}^{\gamma}\int_{\hn}\frac{u^2}{r^{4m-2\gamma}}\dv\bigg]\\&= \bigg(\frac{N-1}{2}\bigg)^{4m}\int_{\hn}u^2\dv + \sum_{i=1}^{2m} C_{2m,0}^{i}\int_{\hn}\frac{u^2}{r^{2i}}\dv.
	\end{align*}
	
	 Substituting the value of $\Xi_{2,1}^2$ and in the end changing index from $j$ to $j-1$, we obtain
	 \begin{align*}
	 C_{2m,0}^1&=\frac{N(N-1)}{16}\bigg(\frac{N-1}{2}\bigg)^{4m-4}+\frac{N(N-1)}{2^{4m-4}}\sum_{j=1}^{m-1}(N-1)^{4m-2j-6} \ \Xi_{2,1}^{2}\\&= \frac{N(N-1)}{2^{4m}}\sum_{j=0}^{m-1}(N-1)^{4m-2j-4}=\frac{N(N-1)}{2^{4m}}\sum_{j=1}^{m}(N-1)^{4m-2j-2}
	 \end{align*}
	 and finally arranging the terms after plugging in the value of $\Xi_{4m-4,1}^{0}$, we deduce
	 \begin{align*}
	 C_{2m,0}^{2m}  =\bigg(\frac{N-4}{2^{2m-2}}\bigg)^2\ \prod_{j=1}^{m-2}(N+4j)^2(N-4j-4)^2 \ \Xi_{4m-4,1}^{0} =\bigg(\frac{N-4}{2^{2m}}\bigg)^2\ \prod_{j=1}^{m-1}(N+4j)^2(N-4j-4)^2.
	 \end{align*}
	 	 
	 This gives that inequality holds for $k=2m$ and completes the induction.
	 
	  Next we turn to the case $k=2m+1$ odd with the same idea to argue by induction on $m$. Notice, if $m=0$, \eqref{even_and_odd} follows directly from \eqref{r10} with $C_{1,0}^1=1/4$. Next in a similar manner, assuming result is true for $k=2m-1\geq 1$, we can extend it for the case $k=2m+1$, by applying Lemma \ref{lemma_6} and \eqref{r20} suitably. For the brefity we are skipping the details.
	\end{proof}

\begin{remark}
	Using Corollary \ref{cor_1} for $\alpha=2$ in \eqref{r21d} we deduce that for $u\in\cchnm$, with $N\geq 7$ there holds  
	\begin{align}\label{dr21}
	\int_{\hn} |\Delta_{r} u|^2 \ {\rm d}v_{\hn} -   \left( \frac{N-1}{2} \right)^{2} \int_{\hn} |\nabla_{r} u|^2  \ {\rm d}v_{\hn}
	\geq \frac{(N-4)^2}{16} \int_{\hn} \frac{u^2}{r^{4}} \ {\rm d}v_{\hn}+\frac{(N-1)}{8}\int_{\hn} \frac{u^2}{r^{2}} \ {\rm d}v_{\hn}.
	\end{align}
\end{remark}

\begin{remark}
	Exploiting \eqref{r10} in the above inequality \eqref{dr21}, we deduce for all $u\in\cchnm$, with $N\geq 7$ there holds
	\begin{align}\label{dr20}
	\int_{\hn} |\Delta_{r} u|^2 \ {\rm d}v_{\hn} -  \left( \frac{N-1}{2} \right)^{4} \int_{\hn} u^2  \ {\rm d}v_{\hn}  
	\geq \frac{(N-4)^2}{16} \int_{\hn} \frac{u^2}{r^{4}} \ {\rm d}v_{\hn}+\frac{(N^2-1)}{16}\int_{\hn} \frac{u^2}{r^{2}}  \ {\rm d}v_{\hn}.
	\end{align}
	If we compare \eqref{dr21} and \eqref{dr20} with the inequalities \eqref{r21} and \eqref{r20} respectively, then it is easy to observe that inequalities in the first case perform better when $r$ approaching towards zero. In particular, this creates another interesting fact that if we compare \eqref{dr20}, after applying \cite[Lemma 6.1]{EGR} suitably, with \cite[Theorem 2.3]{EGR} in the manifold $M=\hn$ with $N\geq 7$, then the constant appearing in front of the Hardy term $\frac{1}{r^2}$ can be larger than $\frac{(N-1)^2}{16}$ as proved in \cite{EGR}, keeping the constant in front of Rellich term unchanged. Also, we notice that unfortunately in both cases finding the best possible constant is still an open question.
\end{remark}

If we use above inequality \eqref{dr20}, Lemma \ref{lemma_7} and \eqref{r10}, then we will be obtaining the following corollary, where the constants are larger than \eqref{even_and_odd} but demand more dimensional restriction.

\begin{corollary}\label{cor_4.1}
		Let $k$ be a positive integer and $N\geq 4k-1$. Then there exist $k$ positive constants $D_{k,0}^i$ such that the following inequality holds
		\begin{equation}\label{deven_and_dodd}
		\int_{\hn} |\gr^k u|^2\dv - \bigg(\frac{N-1}{2}\bigg)^{2k}\int_{\hn}u^2\dv
		\geq \sum_{i=1}^{k}D_{k,0}^{i}\int_{\hn}\frac{u^2}{r^{2i}}\dv,
		\end{equation}
		for all $u\in\cchnm$. Moreover, the leading terms are given by $D_{k,0}^k= C_{k,0}^{k}$ and 
		\begin{equation*}
		D_{k,0}^1=
		\begin{dcases}
		\frac{(N^2-1)}{16}\sum_{j=1}^{m}\bigg(\frac{N-1}{2}\bigg)^{4m-2j-2} & \text{ if } k=2m, \\
		(N^2-1)\sum_{j=1}^{m}\frac{(N-1)^{2m+2j-2}}{2^{2m+2j+2}} +\frac{(N-1)^{2m}}{2^{2m+2}}  & \text{ if } k=2m+1. \\
		\end{dcases}
		\end{equation*}
\end{corollary}

\subsection{General case $k=2m$ even and $l=2h$ even} 
\begin{theorem}
	Let $k=2m>l=2h\geq 0$ be integers and $N>2k$. Then there exist $k$ positive constants $C_{k,l}^i$ such that for all $u\in \cchnm$, there holds
	\begin{equation}\label{even_even}
	\int_{\hn} (\lr^m u)^2\dv - \bigg(\frac{N-1}{2}\bigg)^{4(m-h)}\int_{\hn}(\lr^h u)^2\dv
	\geq \sum_{i=1}^{k}C_{k,l}^{i}\int_{\hn}\frac{u^2}{r^{2i}}\dv,
	\end{equation}
	where $C_{k,l}^{k}=C_{k-l,0}^{k-l} \ \Xi_{2(k-l),l/2}^0$ and $C_{k,l}^1=C_{k-l,0}^{1} \ \Xi_{2,l/2}^l$.
\end{theorem}
\begin{proof}
	By applying first \eqref{even_and_odd} with $k=2(m-h)$, then \eqref{lemma_6} with $\alpha=2i$ and $\beta=h$, we deduce
	\begin{align*}
	& \int_{\hn}(\lr^m u)^2\dv=\int_{\hn}(\lr^{m-h}(\lr^h u) )^2\dv\\ &\geq \bigg(\frac{N-1}{2}\bigg)^{4(m-h)}\int_{\hn}(\lr^h u)^2\dv
	+ \sum_{i=1}^{2(m-h)}C_{2(m-h),0}^{i}\int_{\hn}\frac{(\lr^h u)^2}{r^{2i}}\dv\\& \geq
	\bigg(\frac{N-1}{2}\bigg)^{4(m-h)}\int_{\hn}(\lr^h u)^2\dv + \sum_{i=1}^{2(m-h)}C_{2(m-h),0}^{i}	\bigg[\sum_{j=0}^{2h}\Xi_{2i,h}^j\int_{\hn}\frac{u^2}{r^{2i+4h-2j}}\dv\bigg]\\& =\bigg(\frac{N-1}{2}\bigg)^{4(m-h)}\int_{\hn}(\lr^h u)^2\dv+\sum_{i=1}^{2m}C_{2m,2h}^{i}\int_{\hn}\frac{u^2}{r^{2i}}\dv
	\end{align*}
	
	with $C_{2m,2h}^1=C_{2(m-h),0}^1 \ \Xi_{2,h}^{2h}$ and $C_{2m,2h}^{2m}=C_{2(m-h),0}^{2(m-h)} \ \Xi_{4(m-h),h}^{0}$.
	\end{proof}

\subsection{General case $k=2m+1$ odd and $l=2h$ even}
\begin{theorem}
	Let $k=2m+1>l=2h\geq 0$ be integers and $N> 2k$. Then there exist $k$ positive constants $C_{k,l}^i$ such that for all $u\in \cchnm$, there holds
	\begin{equation}\label{odd_even}
	\int_{\hn} |\gr(\lr^m u)|^2\dv - \bigg(\frac{N-1}{2}\bigg)^{4(m-h)+2}\int_{\hn}(\lr^h u)^2\dv
	\geq \sum_{i=1}^{k}C_{k,l}^{i}\int_{\hn}\frac{u^2}{r^{2i}}\dv,
	\end{equation}
	 where $C_{k,l}^{k}=C_{k-l,0}^{k-l} \ \Xi_{2(k-l),l/2}^0$ and $C_{k,l}^1=C_{k-l,0}^{1} \ \Xi_{2,l/2}^l$.
\end{theorem}
\begin{proof}
	Exploiting \eqref{even_and_odd} with $k=2(m-h)+1$ and Lemma \ref{lemma_6} for $\alpha=2i$ and $\beta=h$, we obtain
	\begin{align*}
	& \int_{\hn}|\gr(\lr^m u)|^2\dv=\int_{\hn}|\gr^{2(m-h)+1} (\lr^h u) |^2\dv\\ &\geq \bigg(\frac{N-1}{2}\bigg)^{4(m-h)+2}\int_{\hn}(\lr^h u)^2\dv
	+ \sum_{i=1}^{2(m-h)+1}C_{2(m-h)+1,0}^{i}\int_{\hn}\frac{(\lr^h u)^2}{r^{2i}}\dv\\& \geq
	\bigg(\frac{N-1}{2}\bigg)^{4(m-h)+2}\int_{\hn}(\lr^h u)^2\dv + \sum_{i=1}^{2(m-h)+1}C_{2(m-h)+1,0}^{i}	\bigg[\sum_{j=0}^{2h}\Xi_{2i,h}^j\int_{\hn}\frac{u^2}{r^{2i+4h-2j}}\dv\bigg]\\& =\bigg(\frac{N-1}{2}\bigg)^{4(m-h)+2}\int_{\hn}(\lr^h u)^2\dv+\sum_{i=1}^{2m+1}C_{2m+1,2h}^{i}\int_{\hn}\frac{u^2}{r^{2i}}\dv
	\end{align*}
	with $C_{2m+1,2h}^{2m+1}=C_{2(m-h)+1,0}^{2(m-h)+1} \ \Xi_{4(m-h)+2,h}^0$ and $C_{2m+1,2h}^1=C_{2(m-h)+1,0}^{1} \ \Xi_{2,h}^{2h}$.	
\end{proof}

\subsection{General case $k=2m$ even and $l=2h+1$ odd}
\begin{theorem}
	Let $k=2m>l=2h+1\geq1$ be integers and $N>2k$. Then there exist $k$ positive constants $C_{k,l}^i$ such that for all $u\in\cchnm$, there holds 
	\begin{equation}\label{even_odd}
	\int_{\hn} (\lr^{m} u)^2\dv - \bigg(\frac{N-1}{2}\bigg)^{4(m-h)-2}\int_{\hn}|\gr(\lr^h u)|^2\dv
	\geq \sum_{i=1}^{k}C_{k,l}^{i}\int_{\hn}\frac{u^2}{r^{2i}}\dv,
	\end{equation}
	where 
	\begin{equation*}
	C_{k,l}^1=
	\begin{dcases}
	\frac{(N-1)^{2k-2l-1}}{2^{2k-2l+2}} \ \Xi_{2,(l-1)/2}^{l-1}+C_{k-l-1,0}^{1} \ \Xi_{2,(l+1)/2}^{l+1} & \text{ if } k=2m, \  l=2h+1 \text{ and }m-h\neq 1, \\
	\frac{(N-1)}{16} \ \Xi_{2,(l-1)/2}^{l-1}  & \text{ if } k=2h+2 \text{ and }l=2h+1, \\
	\end{dcases}
	\end{equation*} and
	\begin{equation*}
	C_{k,l}^k=
	\begin{dcases}
	C_{k-l-1,0}^{k-l-1} \ \Xi_{2(k-l-1),(l+1)/2}^0 & \text{ if } k=2m, \  l=2h+1 \text{ and }m-h\neq 1, \\
	\frac{(N-4)^2}{16} \ \Xi_{4,(l-1)/2}^0  & \text{ if } k=2h+2 \text{ and }l=2h+1. \\
	\end{dcases}
	\end{equation*}
\end{theorem}
\begin{proof}
	Let $m-h\neq1$. By applying first \eqref{even_and_odd} with $k=2(m-h-1)$, then \eqref{r21} and in the end Lemma \ref{lemma_6} with $\alpha = 2i, \beta = h+1$ once and another time with $\alpha=2,\beta=h$, we obtain
	\begin{align*}
	& \int_{\hn}(\lr^m u)^2\dv=\int_{\hn}(\lr^{m-h-1}(\lr^{h+1} u) )^2\dv \\&\geq \bigg(\frac{N-1}{2}\bigg)^{4(m-h-1)}\int_{\hn}(\lr^{h+1} u)^2\dv
	+ \sum_{i=1}^{2m-2h-2}C_{2m-2h-2,0}^{i}\int_{\hn}\frac{(\lr^{h+1} u)^2}{r^{2i}}\dv\\&\geq \bigg(\frac{N-1}{2}\bigg)^{4(m-h-1)}\bigg[\bigg(\frac{N-1}{2}\bigg)^2\int_{\hn}|\gr(\lr^h u)|^2\dv+ \frac{(N-4)^2}{16} \int_{\hn} \frac{(\lr^h u)^2}{r^{4}} \, {\rm d}v_{\hn}\\&+\frac{(N-1)}{16}\int_{\hn} \frac{(\lr^h u)^2}{r^{2}} \, {\rm d}v_{\hn}\bigg] + \sum_{i=1}^{2m-2h-2}C_{2m-2h-2,0}^{i}\bigg[\sum_{j=0}^{2h+2}\Xi_{2i,h+1}^j\int_{\hn}\frac{u^2}{r^{2i+4h+4-2j}}\dv\bigg]\\&\geq \bigg(\frac{N-1}{2}\bigg)^{4(m-h-1)}\bigg[\bigg(\frac{N-1}{2}\bigg)^2\int_{\hn}|\gr(\lr^h u)|^2\dv+ \frac{(N-4)^2}{16} \int_{\hn} \frac{(\lr^h u)^2}{r^{4}} \, {\rm d}v_{\hn}\bigg]\\&+\frac{(N-1)^{4(m-h)-3}}{2^{4(m-h)}}\bigg[\sum_{j=0}^{2h}\Xi_{2,h}^j\int_{\hn}\frac{u^2}{r^{2+4h-2j}} \ {\rm d}v_{\hn}\bigg] \\& + \sum_{i=1}^{2m-2h-2}C_{2m-2h-2,0}^{i}\bigg[\sum_{j=0}^{2h+2}\Xi_{2i,h+1}^j\int_{\hn}\frac{u^2}{r^{2i+4h+4-2j}} \ {\rm d}v_{\hn}\bigg] \\& =\bigg(\frac{N-1}{2}\bigg)^{4(m-h)-2}\int_{\hn}|\gr(\lr^h u)|^2\dv+\sum_{i=1}^{2m}C_{2m,2h+1}^{i}\int_{\hn}\frac{u^2}{r^{2i}}\dv.
	\end{align*}
	Furthermore, one can observe
	\begin{align*}
	C_{2m,2h+1}^1= \frac{(N-1)^{4(m-h)-3}}{2^{4(m-h)}} \ \Xi_{2,h}^{2h} + C_{2m-2h-2,0}^{1} \ \Xi_{2,h+1}^{2h+2} \text{ and } C_{2m,2h+1}^{2m}=C_{2m-2h-2,0}^{2m-2h-2} \ \Xi_{4m-4h-4,h+1}^0
	\end{align*}
	 and this establishes the result.

	If $m-h=1$, then exploiting \eqref{r21} and Lemma \ref{lemma_6} with $\alpha=4,\beta=h$ once and then $\alpha=2,\beta=h$, \eqref{even_odd} holds with proper constants and this concludes the proof.
\end{proof}

\subsection{General case $k=2m+1$ odd and $l=2h+1$ odd}
\begin{theorem}
	Let $k=2m+1>l=2h+1\geq1$ be integers and $N>2k$. Then there exist $k$ positive constants $C_{k,l}^i$ such that for all $u\in\cchnm$, there holds 
	\begin{equation}\label{odd_odd}
	\int_{\hn} |\gr(\lr^m u)|^2\dv - \bigg(\frac{N-1}{2}\bigg)^{4(m-h)-2}\int_{\hn}|\gr(\lr^h u)|^2\dv
	\geq \sum_{i=1}^{k}C_{k,l}^{i}\int_{\hn}\frac{u^2}{r^{2i}}\dv,
	\end{equation}
	where $C_{k,l}^k=\frac{1}{4} \ \Xi_{2,(k-1)/2}^{0}$ and
	\begin{equation*}
	C_{k,l}^1=
	\begin{dcases}
	\frac{1}{4} \ \Xi_{2,(k-1)/2}^{k-1}+ \frac{(N-1)^2}{4}C_{k-1,l}^{1} & \text{ if } k=2m+1, \  l=2h+1 \text{ and }m-h\neq 1, \\
	\frac{(N-1)^3}{2^6}  \ \Xi_{2,(l-1)/2}^{l-1}+\frac{1}{4} \ \Xi_{2,(l+1)/2}^{l+1} & \text{ if } k=2h+3 \text{ and }l=2h+1. \\
	\end{dcases}
	\end{equation*}
\end{theorem}
\begin{proof}
	Assume $m-h\neq 1$. Exploiting first \eqref{r10}, then \eqref{even_odd} for the index $k=2m,l=2h+1$ and finally Lemma \ref{lemma_6} with $\alpha=2,\beta=m$, we have 
	\begin{align*}
	\int_{\hn}|\gr(\lr^m u)|^2\dv & \geq \left( \frac{N-1}{2} \right)^{2} \int_{\hn} (\lr^m u)^2 \ {\rm d}v_{\hn} + \frac{1}{4} \int_{\hn} \frac{(\lr^m u)^2}{r^2} \ {\rm d}v_{\hn}\\&=\bigg(\frac{N-1}{2}\bigg)^2\bigg[ \bigg(\frac{N-1}{2}\bigg)^{4(m-h)-2}\int_{\hn}|\gr(\lr^h u)|^2\dv
	\\&+ \sum_{i=1}^{2m}C_{2m,2h+1}^{i}\int_{\hn}\frac{u^2}{r^{2i}}\dv\bigg]+\frac{1}{4}\sum_{j=0}^{2m}\Xi_{2,m}^j\int_{\hn}\frac{u^2}{r^{2+4m-2j}}\dv \\&  =\bigg(\frac{N-1}{2}\bigg)^{4(m-h)}\int_{\hn}|\gr(\lr^h u)|^2\dv+\sum_{i=1}^{2m+1}C_{2m+1,2h+1}^{i}\int_{\hn}\frac{u^2}{r^{2i}}\dv,
	\end{align*}
	where constants are represented by $C_{2m+1,2h+1}^{2m+1}=\frac{1}{4} \ \Xi_{2,m}^{0}$ and $C_{2m+1,2h+1}^1= \frac{1}{4} \ \Xi_{2,m}^{2m}+ \frac{(N-1)^2}{4}C_{2m,2h+1}^{1}$.
	
	If $m-h=1$, then using first \eqref{r10}, then inequality \eqref{r21} and in the end applying Lemma \ref{lemma_6} with  $\alpha=2,\beta=h+1$ first and then $\alpha=2,\beta=h$, we deduce the results.
\end{proof}

Making use of Lemma \ref{lemma_7}, Corollary \ref{cor_4.1} and improved inequalities for lower indices namely \eqref{r10}, \eqref{dr21} and \eqref{dr20} with the preceding technique, we will be obtaining another version of \eqref{imp_r_high_poin}. This result gives a better constant but requires more dimension restriction. Without detailing the proof, just by noting this result, we will finish this section.

\begin{corollary}
	Let $k>l$ be positive integers and $N\geq 4k-1$. There exist $k$ positive constants such that for all $u\in\cchnm$, there holds 
	\begin{align}\label{d_imp_r_high_poin}
	\int_{\hn}  |\nabla_{r}^{k} u|^2 \, {\rm d}v_{\hn}  - \bigg(\frac{N-1}{2}\bigg)^{2(k-l)} \int_{\hn} |\nabla_{r}^{l} u|^2 \, {\rm d}v_{\hn}\geq \sum_{j=1}^{k}D^{j}_{k,l}\int_{\hn}\frac{u^2}{r^{2j}}\dv.
	\end{align}
	Moreover, the leading terms for $r\rightarrow 0$ and $r \rightarrow \infty$, namely $D^{k}_{k,l}$ and $D^{1}_{k,l}$ are given by as follows
	\begin{equation*}
	D^{k}_{k,l}:=
	\begin{dcases}
	D_{k-l,0}^{k-l} \ \zeta_{2(k-l),l/2}^0 & \text{ if }l=2h \text{ and }k>l,\\
	D_{k-l-1,0}^{k-l-1} \ \zeta_{2(k-l-1),(l+1)/2}^0 & \text{ if } k=2m, \  l=2h+1 \text{ and }m-h\neq 1, \\
	\frac{(N-4)^2}{16} \ \zeta_{4,(l-1)/2}^0  & \text{ if } k=2h+2 \text{ and }l=2h+1, \\
	\frac{1}{4} \ \zeta_{2,(l+1)/2}^0 & \text{ if } k=2m+1 \text{ and }l=2h+1,
	\end{dcases}
	\end{equation*} and 
	\begin{equation*}
	D^{1}_{k,l}:=
	\begin{dcases}
	D_{k-l,0}^{1} \ \zeta_{2,l/2}^l & \text{ if }l=2h \text{ and }k>l,\\
	\frac{(N-1)^{2k-2l-1}}{2^{2k-2l+1}} \ \zeta_{2,(l-1)/2}^{l-1}+D_{k-l-1,0}^{1} \ \zeta_{2,(l+1)/2}^{l+1} & \text{ if } k=2m, \  l=2h+1 \text{ and }m-h\neq 1, \\
	\frac{(N-1)}{8} \ \zeta_{2,(l-1)/2}^{l-1}  & \text{ if } k=2h+2 \text{ and }l=2h+1, \\
	\frac{1}{4} \ \zeta_{2,(k-1)/2}^{k-1}+ \frac{(N-1)^2}{4}D_{k-1,l}^{1}& \text{ if } k=2m+1, \  l=2h+1 \text{ and }m-h\neq 1, \\
	\frac{(N-1)^3}{2^5}  \ \zeta_{2,(l-1)/2}^{l-1}+\frac{1}{4} \ \zeta_{2,(l+1)/2}^{l+1}  & \text{ if } k=2h+3 \text{ and }l=2h+1.
	\end{dcases}
	\end{equation*}
\end{corollary}

\medskip 

{\bf Acknowledgements.} The author would like to thank Professor E. Berchio for carefully reading the manuscript and for some constructive comments which improved the presentation of the paper. The author is also grateful to Professor D. Ganguly for suggesting the problem and useful discussion. This project is supported by the Council of Scientific \& Industrial Research (File no.  09/936(0182)/2017-EMR-I) and by the Ph.D. program at the Indian Institute of Science Education and Research, Pune.

\end{document}